\documentclass[12pt]{article}
\usepackage{amsmath}
\usepackage{amsfonts}
\usepackage{amsthm}
\usepackage{graphicx}
\usepackage{overpic}
\usepackage{a4wide}
\usepackage{amssymb} 
\usepackage{pictex}
\usepackage{rotating}  
\usepackage{color}
\usepackage{etex} 
\usepackage{enumitem}
\usepackage{accents}

\definecolor{refkey}{rgb}{0,0,1}
\definecolor{labelkey}{rgb}{1,0,0}

\usepackage{verbatim}

\numberwithin{equation}{section}

\newtheorem{theorem}{Theorem}[section]
\newtheorem{proposition}[theorem]{Proposition}
\newtheorem{lemma}[theorem]{Lemma}
\newtheorem{corollary}[theorem]{Corollary}
\newtheorem{Definition}[theorem]{Definition}
\newtheorem*{notation*}{Notation}
\newtheorem*{remark*}{Remark}

\newtheorem{Remark}[theorem]{Remark}
\newenvironment{remark}{\begin{Remark}\rm}{\end{Remark}}
\newtheorem{Example}[theorem]{Example}
\newenvironment{example}{\begin{Example}\rm}{\end{Example}}
\newtheorem{RHproblem}[theorem]{RH problem}

\usepackage{color}

\newcommand{\C}{\mathbb{C}}

\newcommand{\Z}{\mathbb{Z}}
\newcommand{\N}{\mathbb{N}}

\newcommand{\stau}{\overline\tau}

\newcommand{\PP}{\mathbb P}

\renewcommand{\hat}{\widehat}
\renewcommand{\tilde}{\widetilde}

\newcommand{\calC}{\mathcal{C}}
\newcommand{\calM}{\mathcal{M}}

\newcommand{\lot}{\text{\sc lot}}
\newcommand{\calN}{\mathcal{N}}
\newcommand{\calZ}{\mathcal{Z}}
\newcommand{\calB}{\mathcal{B}}
\newcommand{\calS}{\mathcal{S}}

\newcommand{\VDM}{\mathrm{VDM}}

\newcommand{\bV}{\mathbf{V}}
\newcommand{\bv}{\mathbf{v}}

\newcommand{\bb}{\mathbf{b}}

\usepackage[bookmarksopen, naturalnames]{hyperref}

\begin{document}
 
\title{Pluripotential theory on algebraic curves}
\author{Norm Levenberg and Sione Ma'u}
\maketitle 
\begin{abstract} In previous works, the second author defined directional Robin constants associated to a compact, nonpolar subset $K$ of an algebraic curve $A$ in $\C^N$ and related these to a natural class of Chebyshev constants for $K$. We define a second class of Chebyshev constants for $K$; relate these two classes; and utilize each of them to define two families of extremal-like functions which can be used to recover the Siciak-Zaharjuta extremal function for $K$.

\end{abstract}

\section{Introduction} In \cite{BM}, \cite{HM} and \cite{M}, the authors defined directional Robin constants associated to a compact, nonpolar subset $K$ of an algebraic curve $A$ in $\C^N$ and related these constants to certain Chebyshev constants for $K$. In this work we continue this investigation and delve more deeply into these -- and other -- relations. In this introductory section, we describe the precise setting and state our main results. 

Let $A$ be an irreducible algebraic curve in $\C^N$. We let  $SH(A)$ denote the weakly subharmonic (shm) functions on $A$: $u\in SH(A)$ if $u$ is uppersemicontinuous (usc) on $A$ and shm on $A^0$, the regular points of $A$. Now let $K\subset A$ be nonpolar, i.e., $K\cap A^0$ is not polar as a subset of the complex curve $A^0$. We can consider the Siciak-Zaharjuta extremal function 
$$V_K(z):=\sup \{\frac{1}{deg(p)}\log |p(z)|: ||p||_K:=\max_{\zeta \in K}|p(\zeta)| \leq 1, \ p \ \hbox{holomorphic polynomial}\}$$
for $z\in \C^N$. 
From Sadullaev's theorem \cite{S}, we know that 
$V_K|_A$ is locally bounded on $A$; and for $z\in A$, 
$$V_K|_A^*(z):=\limsup_{\zeta \to z, \ \zeta \in A}V_K(\zeta)$$
is in $SH(A)$ and it is harmonic on $A^0\setminus K$. Moreover, defining 
$$L(A):=\{u \in SH(A): u(z) \leq \log^+|z|+c_u=\max[0,\log |z|]+c_u\}$$
where $c_u$ is a constant depending on $u$, $V_K|_A=V_{K,A}$ on $A^0$ where
$$V_{K,A}(z)=\sup \{u(z): u\in L(A), \ u\leq 0 \ \hbox{on} \ K\},$$
and $V_{K,A}^*\in L^+(A)$ where 
$$L^+(A):=\{u \in L(A): \log^+|z|-c_u \leq u(z) \leq \log^+|z|+c_u\}.$$
For simplicity, we simply write $V_K:=V_{K,A}=V_K|_A$ and $V_K^*:=V_K|_A^*$ as we only consider points $z\in A$. We say $K$ is {\it regular} if $V_K$ is continuous on $A^0$. As an example, if we choose a basis for $\C^N$ so that 
$$A\subset \{(z_1,z'): |z'|^2  < C(1+|z_1|^2)\}$$
where $z'=(z_2,...,z_{N})$ and $C>0$, then for $K:=\{z\in A: |z_1|\leq r\}$ we have $V_K(z)=\log^+|z_1|/r:=\max[0,\log |z_1|/r]$ (cf., \cite{S}, p. 497). 

Now 
suppose $A=\{z\in \C^N:P(z)=0\}$ where $P$ is an irreducible polynomial of degree $d$ and that $\pi:A\subset \C^N \to \C$ via $\pi(z)=z_1$ is a $d-$sheeted covering map of $A$ over $\C\setminus V$ where $V$ is a finite set of points. Following \cite{BM}, \cite{HM} and \cite{M}, we assume $A$ satisfies the following condition:

{\it The algebraic curve $A$ has $d$ distinct non-parallel linear asymptotes $L_1,...,L_d$; and each $L_j$ may be parameterized by $t\to {\bf c_j}+t{\bf \lambda_j}, \ t\in \C$, where 
${\bf c_j}=(c_{j1},...,c_{jN}), \ {\bf \lambda_j}=(1,\lambda_{j2},...,\lambda_{jN})$ and $\lambda_{jm}\ne \lambda_{km}$ if $j\not = k$ for $m=2,...,N$.}

\noindent Here, a linear asymptote of $A$ is a complex line $L\subset \C^N$ with 
$$\lim_{|z|\to \infty, \ z\in L} |z-z_A|=0$$
where $z_A$ is a closest point to $L$ contained in $H_z\cap A$ where $H_z$ is the complex hyperplane orthogonal to $L$ through $z$. We call $\{{\bf \lambda_j}\}_{j=1,...,d}$ the set of directions of $A$.

For example, if $N=2$ we can write $A=\{(z_1,z_2) \in \C^2: P(z_1,z_2)=0\}$ where 
$$P(z_1,z_2)=\sum_{j=0}^d h_j(z_1,z_2)$$
is an irreducible polynomial of degree $d$, and the sum is a decompositon of $P$ into its homogeneous components $h_j$ of degree $j$. In particular,
$$A_h :=\{(z_1,z_2)\in \C^N: h_d(z_1,z_2))=0\}$$
is the associated homogeneous variety, and we can write
$$h_d(z_1,z_2)=C\prod_{j=1}^d (z_2 - \lambda_j z_1).$$
Letting $\tilde L_j =\{(z_1,z_2):z_2 - \lambda_j z_1=0\}$, we have 
$A_h=\cup_j \tilde L_j$ and these lines $\tilde L_j$ are parallel to the linear asymptotes for $A$. Then $A$ satisfies the italicized condition above precisely when the $\lambda_j$ are distinct and nonzero.

We record some relevant items from \cite{HM}:

\begin{proposition} \label{hmstuff} Given $\epsilon >0$, there exists $R=R(\epsilon)>0$ and $B=B(R)=\{z\in \C^N: |z|<R\}$ such that
\begin{enumerate}
\item $A\setminus \bar B\subset A^0$;
\item $A\setminus \bar B=D_1\cup \cdots \cup D_d$ where $D_1,...,D_d$ are pairwise disjoint domains in $A$;
\item for each $j=1,...,d$, $dist(D_j,L_j)< \epsilon$; 
\item for each $j=1,...,d$, the projection $\pi_j: D_j \to \C$ given by $\pi_j(z_1,z')=z_1$ is one-to-one.

\end{enumerate}

\end{proposition}

From item 4), 
$$A\setminus \bar B=\{(z_1,z'): (z_1,s_j(z_1)), \ j=1,...,d\}$$
where the $s_j(z_1)$ are distinct for $|z_1|>R$. Thus if $u\in SH(A)$, $u$ can be thought of as $d$ shm functions $u_1,...,u_d$ on each ``branch'' 
$$A(j):=\{(z_1,s_j(z_1)): |z_1|>R \}.$$
In particular, given $u\in L(A)$, we define $d$ Robin constants associated to $u$:
\begin{equation} \label{robindef} \rho_{u(j)}:= \limsup_{|z|\to \infty, \ z \in A(j)} [u(z) - \log |z|].
\end{equation}  
Then, in the notation above, we have 
$$\rho_{u(j)}=\limsup_{|z|\to \infty, \ z\in D_j}[u(z)-\log |z_1|]=\limsup_{|z_1|\to \infty, \ z \in A, \ z/z_1\to {\bf \lambda_j}}[u(z)-\log |z_1|].$$
We write $\C[A]$ for the coordinate ring of $A$, i.e., $\C[A]=\C[z]/\textlangle P(z)\textrangle$, and for a polynomial $p$, its degree $deg(p)$ will refer to its degree as an element of $\C[A]$. In particular, for a polynomial $p_n$ of degree $n$ in $\C[A]$, the Robin constants associated to $u:=\frac{1}{n} \log |p_n|$ can be computed as 
\begin{equation}\label{robpoly}  \rho_{u(j)}=\frac{1}{n} \log |\hat p_n({\bf{\lambda_j}})|,  \ j=1,...,d \end{equation}
where $\hat p_n$ is the top degree ($n$) homogeneous part of $p_n$. 
For $K\subset A$ nonpolar, since $V_K^*$ is a harmonic function on $A^0\setminus K$, the limits 
$$\rho_K(\lambda_j):=\rho_{V_K^*(j)}=\lim_{|z|\to \infty, \ z \in A(j)} [V_K^*(z) - \log |z|], \ j=1,...,d$$
exist. 

\begin{example} As a trivial example, when $A\subset \{(z_1,z'): |z'|^2  < C(1+|z_1|^2)\}$
where $z'=(z_2,...,z_{N})$ and $C>0$, recall for $K:=\{z\in A: |z_1|\leq r\}$ we have $V_K(z)=\max[0,\log |z_1|/r]$. Hence 
$$\rho_K(\lambda_j)=\limsup_{|z_1|\to \infty, \ z \in A, \ z/z_1\to {\bf \lambda_j}}[\log |z_1|/r-\log |z_1|]=-\log r$$
for $j=1,...,d.$
\end{example}

Our first goal is to prove the following result.

\begin{theorem} \label{mainth} Let $A$ be an irreducible algebraic curve in $\C^N$ and let $K\subset A$ be nonpolar. Let $u\in L(A)$ with 
$u\leq 0$ on $K$ and suppose $\rho_{u(j)}= \rho_K(\lambda_j), \ j=1,...,d$. Then $u=V_K^*$ on $A^0\setminus K$.
\end{theorem}

In the next section, we give a proof of this result in $\C$ which can be modified to prove the actual result; that proof is then given in section 3. 
Using Theorem \ref{mainth}, Theorem 2.1 of \cite{B98} generalizes to the algebraic curve setting:

\begin{theorem} \label{b98} Let $A$ be an irreducible algebraic curve in $\C^N$ and let $K\subset A$ be nonpolar. Let 
$\{p_n\}$ be a sequence of polynomials with deg$p_n=n$ satisfying $\limsup_{n\to \infty} ||p_n||_K^{1/n}=1$ such that, letting $\hat p_n$ 
be the degree $n$ homogeneous piece of $p_n$,  the function 
$$\hat u(z):=\bigl(\limsup_{n\to \infty}\frac{1}{n}\log |\hat p_n(z)|\bigr)^*\in L(A)$$
satisfies $\rho_{\hat u(j)}= \rho_K(\lambda_j), \ j=1,...,d$. Then 
$$u(z):=\bigl(\limsup_{n\to \infty}\frac{1}{n}\log |p_n(z)|\bigr)^*=V_K^*(z)$$ 
for $z\in A^0\setminus K$.

\end{theorem}

The main goal of this paper is to give some substance to these general results by specializing to compact, nonpolar subsets $K$ of an algebraic curve $A$ in $\C^2$. In section 4, following \cite{HM}, we define classes of Chebyshev polynomials associated to the directions $\lambda_j, \ j=1,...,d$; the corresponding Chebyshev constants $T(K,\lambda_j), \ j=1,...,d$ are related to the Robin constants $\rho_K(\lambda_j)$ of $V_K$. We also consider classes of Chebyshev polynomials associated to a ``standard'' ordering of a monomial basis for $\C[A]$ to obtain Chebyshev constants $T(K,\calZ(j)), \ j=0,...,d-1$. The goal of section 4 is the following result: 

\begin{theorem} \label{chebthm} Suppose the directions $\lambda_1,\ldots,\lambda_d$ are labelled so that 
$$T(K,\lambda_1)\geq T(K,\lambda_2) \geq\cdots\geq T(K,\lambda_d) .$$ Then  
$T(K,\calZ(k-1))=T(K,\lambda_{k})$ for each $k\in\{1,\ldots,d\}$.  
\end{theorem}

\noindent We mention that a key part of the proof involves showing that the limit in the definition of $T(K,\calZ(k))$ exists; here we utilize notions of transfinite diameter for $K$.

This leads, in section 5, to definitions of two families of $d$ extremal-like functions in $L(A)$, denoted $V_{K}^{(k)}, k=1,...,d$, defined in (\ref{vkk}), and $\tilde V_{K}^{(k)}, k=0,...,d-1$, defined in (\ref{vkkt}). With the aid of Theorem \ref{b98}, we prove the following result:

\begin{theorem} For $K\subset A\subset \C^2$ nonpolar, we have
$$\max [V_{K}^{(1)}(z_1,z_2),...,V_{K}^{(d)}(z_1,z_2)]= V_K^*(z_1,z_2)  \ \hbox{on} \ A^0 \setminus K.$$
Furthermore, if  
$$\rho_K(\lambda_1) < \rho_K(\lambda_2) < \cdots < \rho_K(\lambda_d),$$
we have
$$\max [\tilde V_{K}^{(0)}(z_1,z_2),...,\tilde V_{K}^{(d-1)}(z_1,z_2)]= V_K^*(z_1,z_2)  \ \hbox{on} \ A^0 \setminus K.$$

\end{theorem}

We remark that these last two results can be generalized to $\C^N, \ N>2$ but the proofs (and notation) would be cumbersome.  At the end of section 5, we give specific examples to illustrate these results. The bulk of the work in this paper is section 4 where we relate the various Chebyshev constants arising from ordering bases of $\C[A]$ in order to verify the appropriate properties of the extremal-like functions $V_{K}^{(k)}, k=1,...,d$ and $\tilde V_{K}^{(k)}, k=0,...,d-1$.

\section{Proof of Theorem \ref{mainth} in $\C$} Theorem \ref{mainth} is a generalization of a classical potential theoretic result in the plane:

\begin{theorem} \label{inplane} Let $K\subset \C$ be compact and nonpolar with $\C \setminus K$ connected, and let $u\in L(\C)$ satisfy $u\leq 0$ on $K$. Suppose $\rho_u:=\limsup_{|z|\to \infty}[u(z)-\log |z|]$ equals $\rho_K:=\rho_{V_K^*}$. Then $u=V_K^*$ on $\C \setminus K$.

\end{theorem} 

\begin{remark} There is an elementary proof of this fact, but not one that will easily generalize to the  algebraic curve situation. Namely, the function $w:=u-V_K^*$ is shm and nonpositive on $\C\setminus K$ with $\limsup_{|z|\to \infty} w(z)=0$. Considering the Kelvin transform $\tilde w(z) = w(z/|z|^2)$ gives us a shm and nonpositive function on $D\setminus \{0\}$ where $D$ is a bounded domain containing the origin. Since $\tilde w$ is bounded near the origin $0$, it extends across $0$ via $\tilde w(0):=\limsup_{z\to 0} \tilde w(z) \leq 0$ as a shm function on all of $D$. But $\limsup_{|z|\to \infty} w(z)=0$ implies, in fact, that $\tilde w(0)=0$ which contradicts the maximum principle unless $\tilde w\equiv 0$ in $D$, i.e., $w\equiv 0$ in $\C\setminus K$. We thank Franck Wielonsky for this observation.
\end{remark}

We consider shm functions $u \in L^+(\C)$ where
$$L^+(\C)=\{u\in SH(\C): \log^+|z|-c_u \leq u(z) \leq \log^+|z|+c_u\}$$
where $c_u$ is a constant depending on $u$. We will need the following result for such functions on the way to a proof of Theorem \ref{inplane} which will generalize to algebraic curves.  

\begin{proposition} \label{robinprop} Let $u\in L^+(\C)$. Then
$$\rho_u:=\limsup_{|z|\to \infty}[u(z)-\log |z|] =\lim_{R\to \infty}  [\frac{1}{2\pi}\int_0^{2\pi}u(Re^{i\theta})d\theta -log R].$$
\end{proposition}

\begin{proof} Let $w(z):=u(z)-\log |z|$. Then $w$ is shm in $\C \setminus \{0\}$ and bounded near $\infty$. Define $\tilde w(z):=w(z/|z|^2)$. Then $\tilde w$ is shm in a deleted neighborhood of the origin, and is bounded near $\{0\}$. Thus we can extend $\tilde w$ to a shm function in this neighborhood of $\{0\}$ by defining 
$$\tilde w(0):=\limsup_{z\to 0} \tilde w(z)= \rho_u.$$
Then 
$$\rho_u=\tilde w(0) =\lim_{r\to 0^+} \frac{1}{2\pi}\int_0^{2\pi}\tilde w(re^{i\theta})d\theta$$
$$=\lim_{r\to 0^+} \frac{1}{2\pi}\int_0^{2\pi}[u(re^{i\theta}/r^2)-\log 1/r]d\theta= \lim_{R\to \infty}  [\frac{1}{2\pi}\int_0^{2\pi}u(Re^{i\theta})d\theta -log R].$$

\end{proof}

The key fact we need for this second proof of Theorem \ref{inplane} is the following.

\begin{lemma} \label{flatclemma} Let $u,v\in L^+(\C)$. Then 
$$\int_{\C} [udd^cv-vdd^cu]:=\int_{\C} [u \Delta v- v \Delta u]dA = \rho_u -\rho_v.$$
\end{lemma}

\noindent Here, $dd^cu=\Delta udA$ where $dA =\frac{i}{2}dz \wedge d\bar z=dx \wedge dy$ and $\Delta u$ are distributional derivatives.

\begin{proof} 

We may assume $u,v$ are smooth. It suffices to prove the result for $u\in L^+(\C)\cap C^2(\C)$ and $u_0(z):=\log^+|z|$ for then we apply the result to $v\in L^+(\C)\cap C^2(\C)$ and $u_0$ and hence to $u,v$. For $R>>1$, letting $B_R:=\{z: |z| <R\}$, we have 
$$\int_{B_R} [udd^cu_0-u_0dd^cu]=\int_{\partial B_R} [ud^cu_0-u_0d^cu]=\int_{\partial B_R} [u \frac{\partial u_0}{\partial n}- u_0\frac{\partial u}{\partial n}]ds$$
$$\int_{\partial B_R} [u /R- \log R \frac{\partial u}{\partial n}]ds=\frac{1}{2\pi} \int_0^{2\pi}[u(Re^{i\theta}) /R- \log R \frac{\partial u}{\partial n}]Rd\theta$$
$$=\frac{1}{2\pi} \int_0^{2\pi} u(Re^{i\theta})d\theta -\log R \int_{\partial B_R} \frac{\partial u}{\partial n}ds$$
$$=\frac{1}{2\pi} \int_0^{2\pi} u(Re^{i\theta})d\theta -\log R \int_{B_R} \Delta u dA.$$
Letting $n_u(R):= \int_{B_R} \Delta u dA$, we have
$$\int_{B_R} [udd^cu_0-u_0dd^cu]= \frac{1}{2\pi} \int_0^{2\pi} u(Re^{i\theta})d\theta -\log R \cdot n_u(R)$$
$$=\frac{1}{2\pi} \int_0^{2\pi} u(Re^{i\theta})d\theta -\log R + [1-n_u(R)]\log R.$$
But by Corollary 1.1 of \cite{L}, for $u\in L^+(\C)$, $\lim_{R\to \infty}[1-n_u(R)]\log R=0$ (or see section 1 of \cite{BLM}, specifically, the proof of Proposition 1.1 and items (i)-(iii) on p. 62). The result follows from Proposition \ref{robinprop}.

\end{proof}

Next, we recall a generalized comparison theorem (stated in the version we will use; the conclusion is true under slightly weaker assumptions). This is Lemma 6.5 in \cite{BT}.

\begin{proposition} \label{compthm} Let $w\in L(\C)$ and $v\in L^+(\C)$. If $w\leq v$ $dd^cv-$a.e., then $w\leq v$ in $\C$.

\end{proposition}

We now use the lemma and proposition to prove Theorem \ref{inplane} as in Lemma 2.1 of \cite{B98}.

\begin{proof} Choose $c>0$ so that $\log|z| < c$ on $K$ and consider $v(z):=\max[u,0,\log |z|-c]$. Then $v\in L^+(\C)$ and $\rho_v:=\limsup_{|z|\to \infty}[v(z)-\log |z|]=\rho_u=\rho_K$. From the lemma, 
$$\int_{\bf C} V_K^* dd^c v =  \int_{\bf C} v dd^cV_K^*.$$
Since $v=0$ on $K$, the right-hand-side is $0$. But $V_K^*>0$ outside of $K$ so $dd^c v$ puts no mass on $\C \setminus K$. Hence 
$V_K^* \leq v$ $dd^cv-$a.e. and by Proposition \ref{compthm}, $V_K^* \leq v$ on $\C$. For $z\in \C \setminus K$, $V_K^* >0$ so that 
$V(z)> \log |z| - c$; hence $V_K^*(z)=u(z)$ for such $z$.

\end{proof}

\section{Proof of Theorem \ref{mainth}} As in the introduction, let $A$ be an irreducible algebraic curve in $\C^N$ and let $K\subset A$ be nonpolar. The appropriate generalization of Lemma \ref{flatclemma} in this setting is the following:

\begin{lemma} Let $u,v\in L(A)$ with $\rho_{u(j)}, \rho_{v(j)}$ finite for $j=1,...,d$. Then 
$$\int_{A} [udd^cv-vdd^cu]= \sum_{j=1}^d [\rho_{u(j)}- \rho_{v(j)}].$$
\end{lemma}

\begin{proof} We may assume $u,v$ are smooth on $A^0$. For $R>>1$, let $B_R:=\{z:|z|<R\}$. Using Lemma \ref{hmstuff}, we can fix $R_0$ so that for $R>R_0$, $\pi:A\cap B_R^c=D_1\cup \cdots \cup D_d \to \{z_1\in \C: |z_1|>R\}$ is an unramified $d$ to $1$ cover. Then for $j=1,...,d$, 
$$\int_{\partial D_j} [ud^cv-vd^cu]= \int_{\{|z_1|=R\}}  [u_j d^c v_j -v_j d^c u_j]$$ where $u_j(z_1):= u(\pi_j(z_1,z'))$ and $\pi_j$ is the projection $\pi$ restricted to $D_j$. Note that $u_j\in L^+(\C)$ since $\rho_{u_j}= \rho_{u(j)}$ is finite. We refer the reader to section 3, lemmas 3.2 and 3.3, of \cite{HM}, for justification. As in the proof of Lemma \ref{flatclemma}, it suffices to consider the case where  $v(z_1,z')=\log^+|z_1|$; then 
$$\int_{A\cap B_R} [udd^cv-vdd^cu]=\int_{\partial (A\cap B_R)} [ud^cv-vd^cu]$$
$$=\sum_{j=1}^d \bigl(\int_{\partial D_j} [u((z_1,s_j(z_1))d^c\log |z_1|-\log |z_1|d^cu((z_1,s_j(z_1))]\bigr)$$
$$=\sum_{j=1}^d \bigl(\int_{\{|z_1|=R\}} [u_j(z_1)d^c\log |z_1|-\log |z_1|d^cu_j(z_1)].$$
The proof then proceeds as in Lemma \ref{flatclemma}.

\end{proof}

Using the fact that the proof of Lemma 6.5 in \cite{BT}, i.e., the generalized comparison theorem, goes through with minor modifications to $w\in L(A)$ and 
$v\in L^+(A)$ with $w\leq v$ $dd^cv-$a.e., the proof of the main theorem proceeds as in the previous section:

\begin{proof} Choose $c>0$ so that $\log|z| < c$ on $K$ and consider $v(z):=\max[u,0,\log |z|-c]$. Then $v\in L^+(A)$ and $\rho_{u(j)}= \rho_{v(j)}=\rho_K(\lambda_j)$, $j=1,...,d$. From the lemma, 
$$\int_{A} V_K^* dd^c v =  \int_{A} v dd^cV_K^*.$$
Since $v=0$ on $K$, the right-hand-side is $0$. But $V_K^*>0$ outside of $K$ so $dd^c v$ puts no mass on $A^0 \setminus K$. Hence 
$V_K^* \leq v$ $dd^cv-$a.e. and by the generalized comparison theorem, $V_K^* \leq v$ on $A^0$. For $z\in A^0 \setminus K$, $V_K^* >0$ so that 
$V(z)> \log |z| - c$; hence $V_K^*(z)=u(z)$ for such $z$.

\end{proof}

We now give the proof of Theorem \ref{b98}, following \cite{B98}.

\begin{proof} From the hypothesis that $\limsup_{n\to \infty} ||p||_K^{1/n}=1$, it follows that $u\leq V_K$ and, in particular, $\rho_{u(j)}\leq \rho_K(\lambda_j)=\rho_{\hat u(j)}, \ j=1,...,d$. The only modification of the proof of Theorem 2.1 of \cite{B98} is in the verification that we have equality $\rho_{u(j)}= \rho_K(\lambda_j), \ j=1,...,d$; then the result follows from Theorem \ref{mainth}. For each $j=1,...,d$, the univariate function $v_n(z_1):=\frac{1}{n}\log |p_n(z_1,s_j(z_1))|$ is shm on $D_j$, and hence
$$\limsup_{|z_1|\to \infty, \ z\in D_j} [v_n(z_1)-\log |z_1|]=\inf_{R\geq 1}[\max_{|z_1|=R}v_n(z_1)-\log R].$$
This implies that (recall (\ref{robpoly}))
$$\frac{1}{n}\log |\hat p_n({\bf{\lambda_j}})|\leq \max_{|z_1|=R}\frac{1}{n}\log |p_n(z_1,s_j(z_1))|-\log R.$$
Taking the $\limsup$ as $n\to \infty$,   
$$\limsup_{n\to \infty} \frac{1}{n}\log |\hat p_n({\bf{\lambda_j}})|\leq \max_{|z_1|=R} u(z_1,s_j(z_1))-\log R$$
where we have used Hartogs' lemma and the fact that $u(z)\geq \limsup_{n\to \infty}\frac{1}{n}\log |p_n(z)|$. 
Letting $R\to \infty$,
$$\limsup_{n\to \infty} \frac{1}{n}\log |\hat p_n({\bf{\lambda_j}})|\leq \inf_{R\geq 1} [\max_{|z_1|=R} u(z_1,s_j(z_1))-\log R].$$
Thus, $\rho_{\hat u(j)}\leq \rho_{u(j)}$.

\end{proof}

\section{Chebyshev constants and transfinite diameter}



We begin with some generalities on classes of polynomials in $\C[z]$ or $\C[A]$ and Chebyshev constants associated to compact subsets $K$ of $\C^N$ or $A$. Given homogeneous polynomials $Q$ and $R$, define the polynomial classes 
$$
\calM(Q):= \{ Q^n + \lot\colon n\in \N \},\quad   \calM_R(Q) :=\{ RQ^n+\lot\colon n\in \N\},
$$
where `$\lot$' in `$q+\lot$' stands for \emph{lower order terms}, i.e., a polynomial of degree strictly less than $q$.  For more clarity we might also write $\lot(q)$ or $\lot(n)$ for terms with degrees strictly less than $\deg(q)$ or $n\in \N$.  

We will also consider other classes of polynomials,  but any class $\calN$ will satisfy:%
\begin{itemize}
\item[(i)] $\deg(\calN):=\{\deg(p)\}_{p\in\calN}$ is an unbounded subset of $ \N$;
\item[(ii)] for any compact set $K$ and $n\in\deg(\calN)$, there exists $q\in \calN$ with 
$||q||_K = \inf\{\|p\|_K\colon p\in\calN,\ \deg(p)=n\}$.
\end{itemize}
We will refer to $\calN$ as a (generalized) \emph{monic polynomial class}.  

Given $n\in\deg(\calN)$, define
\begin{equation}\label{eqn:Tndef}
T_n(K,\calN):= \inf\{\|p\|_K\colon p\in\calN,\ \deg(p)=n\}^{1/n}.
\end{equation}
A polynomial in $\calN$ that attains the inf is called a \emph{Chebyshev polynomial} of degree $n$ associated to $K, \calN$.  Then  
define 
\begin{equation}\label{eqn:Tdef}
\overline T(K, \calN):=\limsup_{\substack{n\to\infty\\ n\in\deg(\calN)}} T_n(K, \calN),\quad \underline T(K, \calN):=\liminf_{\substack{n\to\infty\\ n\in\deg(\calN)}} T_n(K, \calN).
\end{equation}
For convenience we suppress `$n\in\deg(\calN)$' which is understood.  
If $\overline T(K, \calN) = \underline T(K, \calN)$, i.e., the limit
\begin{equation}\label{eqn:Tlim}
T(K, \calN)= \lim_{n\to\infty}T_n(K, \calN) \end{equation} 
exists, we call it the \emph{Chebyshev constant} associated to $K, \calN$.  

The next (classical) result follows since $p,q \in \calM(Q)$ imply $pq\in \calM(Q)$.

\begin{proposition}\label{prop:classical}
The limit \eqref{eqn:Tlim} exists when $\calN=\calM(Q)$.
\end{proposition}

We will need to compare Chebyshev constants associated to different monic polynomial classes. To this end, we derive some estimates.   
\begin{lemma}\label{lem:Tinf}
Let $K$ be a compact set and let $\calN_1,\calN_2$ be two monic polynomial classes. 
 Suppose there is a mapping $\Phi\colon\calN_1\to\calN_2$ with the property that 
$$ \limsup_{\deg(p)\to\infty} \left(\frac{ \|\Phi(p)\|_K}{\|p\|_K}\right)^{1/\deg(p)} \leq M,\qquad        
    \liminf_{\deg(p)\to\infty}\frac{\deg(\Phi(p))}{\deg(p)} \geq c    $$
for some constants $M, c>0$.  
 Then
\begin{equation}\label{eqn:Tinf}
\underline T(K,\calN_2)^c \leq M\cdot \underline T(K,\calN_1). 
\end{equation}
\end{lemma}

\begin{proof}
Let $p\in\calN_1$ be a Chebyshev polynomial of order $n=\deg(p)$.  Let $m_n:=\deg(\Phi(p))$.  Then 
\begin{equation*}
T_{m_n}(K,\calN_2)^{m_n}  \leq \frac{\|\Phi(p)\|_K}{ \|p\|_K}    \|p\|_K = \frac{\|\Phi(p)\|_K}{ \|p\|_K} (T_n(K,\calN_1))^n.
\end{equation*}
Taking $n$-th roots, 
\begin{equation}\label{eqn:Tlim2}
T_{m_n}(K,\calN_2)^{m_n/n}\leq\left(\frac{\|\Phi(p)\|_K}{ \|p\|_K}\right)^{1/n} T_n(K,\calN_1).\end{equation} 
Taking a subsequence $n'$ of $n$   such that $T_{n'}(K,\calN_1)\to\underline T(K,\calN_1)$, we have  
$$\limsup_{n'\to\infty}\left(\frac{\|\Phi(p)\|_K}{ \|p\|_K}\right)^{1/n'}T_{n'}(K,\calN_1) \leq M\cdot \underline T(K,\calN_1)$$   on the right-hand side of the inequality.  On the left-hand side, 
$$\liminf_{n'\to\infty}T_{m_{n'}}(K,\calN_2)^{m_{n'}/n'}\geq \liminf_{n\to\infty}T_{m_n}(K,\calN_2)^c\geq \underline T(K,\calN_2)^c.  
$$ 
\end{proof}

The main example is $\Phi(p):=\tilde Rp$ from $\calM(Q)$ to $\calM_R(Q)$ where $\tilde R$ is any polynomial of the form $\tilde R=R+\lot(R)$. It is easy to see that the conditions of Lemma \ref{lem:Tinf} are satisfied.    Also, $m_n = n+\deg(R)$ and the following property also holds. 

\begin{lemma}\label{lem:Tsup}
Fix the notation and hypotheses in the previous lemma.     
 If  $n\mapsto m_n$ is onto for large $n$, i.e.,   
$\deg(\calN_2)\setminus\deg(\Phi(\calN_1))$ is finite,  then  
\begin{equation}\label{eqn:Tsup}
\overline T(K,\calN_2)^c \leq M\cdot \overline T(K,\calN_1)
\end{equation}
\end{lemma}

\begin{proof}
 The additional hypothesis means that a limsup sequence in $\deg(\calN_2)$ for $\overline T(K,\calN_2)$ will eventually coincide with a sequence in $\deg(\Phi(\calN_1))=\{m_n\}$ as $n\to\infty$.  Then 
$$\overline T(K,\calN_2)^c \leq \limsup_{n\to\infty}T_{m_n}(K,\calN_2)^c\leq \limsup_{n\to\infty}T_{m_n}(K,\calN_2)^{m_n/n}.$$
Using \eqref{eqn:Tlim2}, the right-hand side is bounded above by $M\cdot\overline T(K,\calN_1)$.  
\end{proof}

For convenience, we write equations \eqref{eqn:Tinf}, \eqref{eqn:Tsup} in the shorthand 
$$
\overline{\underline T}(K,\calN_2)^c \leq M\cdot\overline{\underline T}(K,\calN_1).
$$

\begin{corollary}\label{cor:3}
Let $K$ be a compact set and let $R_1$, $R_2$, $R$ be homogeneous polynomials. Then 
\begin{enumerate}
\item $\overline{\underline T}(K,\calM_{R_1R_2}(Q))\leq \overline{\underline T}(K,\calM_{R_1}(Q))$.
\item $\overline{\underline T}(K,\calM_{cR}(Q))= \overline{\underline T}(K,\calM_{R}(Q))$ for any constant $c\neq 0$.
\item $\overline{\underline T}(K,\calM_{R}(\lambda Q))= |\lambda|\cdot\overline{\underline T}(K,\calM_{R}(Q))$ for any constant $\lambda\neq 0$.
\item $\overline{\underline T}(K,\calM_{RQ}(Q)) = \overline{\underline T}(K,\calM_{R}(Q))$.
\end{enumerate}
\end{corollary}

\begin{proof}
\begin{enumerate}
\item Apply Lemmas \ref{lem:Tinf} and \ref{lem:Tsup} to the map $p\mapsto R_2p$.
\item Apply part 1 with $R_1=R$, $R_2=c$, then with $R_1= cR$, $R_2=1/c$.  
\item Apply Lemmas \ref{lem:Tinf} and \ref{lem:Tsup} to the map $p\mapsto\lambda^{\deg(p)}p$ from $\calM_{R}(Q))$ to $\calM_{R}(\lambda Q))$, then to its inverse  $p\mapsto\lambda^{-\deg(p)}p$. 
\item By part 1, $\overline{\underline T}(K,\calM_{RQ}(Q)) \leq \overline{\underline T}(K,\calM_{R}(Q))$.  
But also $\calM_{RQ}(Q)\subset\calM_{R}(Q)$.   Hence 
$\overline{\underline T}(K,\calM_{R}(Q)) \leq \overline{\underline T}(K,\calM_{RQ}(Q))$  by definition. \end{enumerate} \end{proof}

\begin{lemma}\label{lem:Tsum}
Let $R_1$ and  $R_2$ be homogeneous polynomials with  $\deg(R_1)=\deg(R_2)$.  Then $$\overline{\underline T}(K,\calM_{R_1+R_2}(Q))\leq  \max\{\overline{\underline T}(K,\calM_{R_1}(Q)), \overline{\underline T}(K,\calM_{R_2}(Q))\}.$$
\end{lemma}
\begin{proof}
Let $r:=\deg(R_1)=\deg(R_2)$.  If $p_1\in\calM_{R_1}(Q)$ and $p_2\in\calM_{R_2}(Q)$ are polynomials of the same degree then
$$\begin{aligned}
p_1+p_2=(R_1Q^{n} +\lot)+ (R_2Q^{n}+\lot) 
= (R_1+R_2)Q^{n} +\lot,  \end{aligned}$$
so $p_1+p_2\in \calM_{R_1+R_2}(Q)$ with the same or lesser degree.  When $p_1$ and $p_2$ are Chebyshev polynomials of degree $n$ associated to $K, \calM_{R_1}(Q)$ and $K, \calM_{R_2}(Q)$, we have
$$\begin{aligned}
T_{n}(K,\calM_{R_1+R_2}(Q))^{n}  \leq \|p_1+p_2\|_K   
&\leq \|p_1\|_K + \|p_2\|_K \\
& \leq T_{n}(K,\calM_{R_1}(Q))^{n} + T_{n}(K,\calM_{R_2}(Q))^{n} \\
& \leq 2\max\{T_{n}(K,\calM_{R_1}(Q)), T_{n}(K,\calM_{R_1}(Q))\}^{n}.
\end{aligned}$$ 
Take the $n$-th root of both sides, then limsup and liminf as $n\to\infty$. 
\end{proof}

\begin{remark*}\rm
The inequalities may be strict, e.g. consider $R_2=-R_1$ in the above lemma.  To get a strict inequality in Corollary \ref{cor:3} part 1, consider a set $K$ contained in an algebraic set and let $R_2$ be a polynomial that vanishes on this set.
\end{remark*}

 
For the rest of this paper, we turn to the situation where $A$ is an irreducible algebraic curve in $\C^2$,  $$A:= \{z=(z_1,z_2)\in \C^2\colon P_d(z)=0\},$$ with $d=\deg(P_d)$ linear asymptotes, none of which are parallel to a coordinate axis. We begin by discussing some algebraic computations in the coordinate ring $ \C[A]$. The reader should see section 3 of \cite{M} for further results. 

  The asymptotic directions are given by nonzero constants $\lambda_1,\ldots,\lambda_d\in\C$, where each linear asymptote of $A$ is of the form $z_2-\lambda_kz_1=c_k$ for some $c_k\in \C$. Let $\hat P_d$ denote the leading homogeneous part of $P_d$; then 
$$\hat P_d(z) = C\prod_{k=1}^d (z_2-\lambda_kz_1)$$ where $C$ is the coefficient of $z_2^d$.  Let  $\bv_1,\ldots,\bv_d$ be the corresponding directional basis polynomials; these are polynomials of degree $d-1$ given by 
$$
\bv_k = \prod_{j\neq k} \frac{z_2-\lambda_jz_1}{\lambda_k - \lambda_j} = \frac{\hat P_d(z)}{C(z_2-\lambda_kz_1)}\cdot \frac{1}{\prod_{j\neq k}(\lambda_k - \lambda_j)}
$$
 so $\bv_j(1,\lambda_k)=\delta_{jk}:=\left\{ \begin{array}{rl} 0 &\hbox{if } j\neq k \\ 1&\hbox{if }j=k \end{array}\right. .$  

Algebraic computation in $ \C[A]$ means that $P(z)=\hat P(z)+\lot =0$ for all $z\in A$. Using this gives
\begin{eqnarray}
 \bv_j(z)\bv_k(z) &=&
 \left\{ \begin{array}{rl}  \lot &\hbox{if } j\neq k\\
 z_1^{d-1}\bv_j(z) +\lot &\hbox{if } j= k   \end{array}\right. ,  \label{eqn:vmult}  \\
\hbox{ and } \bv_j(z)(z_2-\lambda_jz_1) &=& \lot. \label{eqn:vA}
\end{eqnarray}

Define the \emph{directional} basis $\calC$ of $ \C[A]$ by 
$$\begin{aligned}
&1,z_1,z_2,\ldots, z_1^{d-2},\ldots, z_2^{d-2},    \\
&\bv_1,\ldots,\bv_d, z_1\bv_1,\ldots,z_1\bv_d,\ldots, z_1^{d-2}\bv_1,\ldots,z_1^{d-2}\bv_d,\\ 
&\bv_1^2,\ldots,\bv_d^2,\ldots, z_1^{d-2}\bv_1^2,\ldots,z_1^{d-2}\bv_d^2,\\
&\ldots \\
&\bv_1^k,\ldots,\bv_d^k,\ldots, z_1^{d-2}\bv_1^k,\ldots,z_1^{d-2}\bv_d^k,\ldots 
\end{aligned}$$
Basis elements of degree $\leq d-2$ are standard monomials, and basis elements of degree $n$ are of the form $z_1^r\bv_k^q$ where $n=q(d-1)+r$ and $k=1,...,d; \ 0\leq r<d-1$.

\begin{lemma}\label{lem:cjk}
  Let $j\in\{0,\ldots,d-1\}$. Then in $\C[A]$, 
\begin{equation*}
\displaystyle z_1^jz_2^{d-1-j} = \sum_{k=1}^{d} c_{jk}\bv_k +\lot,  \hbox{ with $c_{jk}\neq 0$ for all $k$.}
\end{equation*}
\end{lemma}

\begin{proof}
The formula itself is elementary linear algebra (express $z_1^jz_2^{d-1-j}$ in terms of the basis $\calC$).  We need to show that the coefficients $c_{jk}$ are all nonzero.

  If $c_{jk}=0$ for some $k$, then multiply on both sides by $\bv_k$ to obtain
\begin{equation}\label{eqn:vkdirection}
\bv_kz_1^jz_2^{d-1-j} = \bv_k\left(\sum_{l\neq k} c_{jl}\bv_l +\lot(d-1)\right) = \lot(2d-2)
\end{equation}
using \eqref{eqn:vmult}.    Hence there exists a polynomial  $Q(z_1,z_2)=\bv_kz_1^jz_2^{d-1-j}+\lot$ that is zero on $A$, so $A\subset\bV(Q):=\{Q=0\}$.  
Factoring the leading homogeneous part,  $$\hat Q=\bv_kz_1^jz_2^{d-1-j}=z_1^jz_2^{d-1-j} \prod_{j\neq k} \frac{z_2-\lambda_jz_1}{\lambda_k - \lambda_j}.$$ 

Observe that the algebraic set $\bV(Q)$ cannot have a linear asymptote in the $k$-th direction; it has asymptotes in the other directions, and possibly additional horizontal and vertical asymptotes from the monomial factor.   Therefore no subset of $\bV(Q)$ has the same asymptotic behavior as $A$, contradicting  $A\subset\bV(Q)$.  Thus $c_{jk}\neq 0$.  \end{proof}

The special case 
\begin{equation}\label{eqn:z1_cjk}
z_1^{d-1}=\sum_{k=1}^d\bv_k + \lot
\end{equation}
follows from Lemma \ref{lem:cjk} and the calculation 
$$
z_1^{d-1}\bv_j = \left(\sum_{k=1}^dc_{dk}\bv_k + \lot(d-1)\right)\bv_j = c_{dj}\bv_j^2 + \lot = c_{dj}z_1^{d-1}\bv_j + \lot.
$$
Equating coefficients yields $c_{dj}=1$.  

Using (\ref{eqn:vmult}), (\ref{eqn:vA}), and Lemma \ref{lem:cjk}, we have the following useful result (cf., section 4 of \cite{BM} or Proposition 5.1 of \cite{HM}).

\begin{corollary} \label{polyprop} For $q\in \C[A]$, 
$$q(z) \bv_k (z)= z_1^{deg(q)}\hat q(1,\lambda_k) \bv_k (z) + \lot.$$
\end{corollary}

\bigskip

Let $K\subset A$ be compact.  Recall by Proposition \ref{prop:classical} that the limit \eqref{eqn:Tlim} exists for $\calN=\calM(\bv_k)$.  For $k=1,...,d$, we define
 \begin{equation} \label{dirchebc} T(K,\lambda_k):=T(K,\calM(\bv_k))\end{equation} 
 to be the \emph{directional Chebyshev constant} of $K$ for the direction $\lambda_k$. Note this appears to be different than the directional Chebyshev constant $\tau(K,\lambda_k)$ of $K$ for the direction $\lambda_k$ as defined in Definition 5.3 of \cite{HM} (and earlier in \cite{BM}); there it was shown that $-\log \tau(K,\lambda_k)=\rho_K(\lambda_k)$. In fact, $\tau(K,\lambda_k)$ from those papers coincides with our Chebyshev constant $T(K,\calM_{\bv_k}(z_1))$; and we show, in Corollary \ref{cor:oldT}, that $T(K,\lambda_k)=T(K,\calM_{\bv_k}(z_1))$. 

\begin{proposition}\label{prop:Tz1z2}
Let $j_1,j_2\in \{0,\ldots,d-1\}, \ k \in \{1,\ldots,d\}$ and suppose $j_1+j_2\leq d-2$.  Then 
$$
\overline{\underline T}(K,\calM_{z_1^{j_1}z_2^{j_2}}(\bv_k)) = T(K,\lambda_k).
$$
In particular, the limit \eqref{eqn:Tlim} also exists for the class $\calN =\calM_{z_1^{j_1}z_2^{j_2}}(\bv_k)$.  
\end{proposition}

\begin{proof}
Write  $j:=j_1$; then $j_2\leq d-2-j$.  By Lemma \ref{lem:cjk}, 
$$\begin{aligned}
z_1^{j}z_2^{d-1-j}\bv_k^n + \lot &= \bv_k\left(\sum_{l=1}^{d} c_{jl}\bv_l +\lot\right)\bv_k^{n-1} +\lot \\
& = c_{jk}\bv_k^{n+1}+\lot
\end{aligned}$$
where $c_{jk}\neq 0$.  Therefore $\calM_{z_1^{j}z_2^{d-1-j}}(\bv_k) = \calM_{c_{jk}\bv_k}(\bv_k)$. By Corollary \ref{cor:3},  
$$\begin{aligned}
T(K,\calM(\bv_k)) =\overline{\underline T}(K,\calM_{c_{jk}\bv_k}(\bv_k)) 
&=\overline{\underline T}(K,\calM_{z_1^{j}z_2^{d-1-j}}(\bv_k))  \\
& \leq \overline{\underline T}(K,\calM_{z_1^{j_1}z_2^{j_2}}(\bv_k)) \leq  T(K,\calM(\bv_k)).
\end{aligned}$$
\end{proof}

The proof of Proposition \ref{prop:Tz1z2} is based on $c_{jk}\neq 0$ from Lemma \ref{lem:cjk}, which uses the fact that the leading homogeneous part of  $z_1^{j}z_2^{d-1-j}\bv_k$ does not contain a factor of $z_2-\lambda_k z_1$.  One can prove something a bit more general. We omit the details as we will not use this result in the sequel. 
\begin{proposition}
Let $R(z_1,z_2)$ be a nonzero homogeneous polynomial in $\C[A]$.  Suppose  for some $k\in\{0,\ldots,d-1\}$, $z_2-\lambda_k z_1$ is not a factor of $R$.  Then 
$$
\overline{\underline T}(K,\calM_R(\bv_k))=T(K,\lambda_k),
$$
i.e., the limit \eqref{eqn:Tlim} exists for $\calN =\calM_{R}(\bv_k)$ and equals the right-hand side.  \qed
\end{proposition}

\begin{corollary}\label{cor:oldT}
We have $T(K,\lambda_k)=T(K,\calM_{\bv_k}(z_1))$ for each $k\in\{1,\ldots,d\}$. Hence 
$$-\log T(K,\lambda_k)=\rho_K(\lambda_k), \ k\in\{1,\ldots,d\}.$$
\end{corollary}

\begin{proof}
By applying \eqref{eqn:vmult} the requisite number of times,
\begin{equation}\label{eqn:nml}
z_1^l\bv_k^m +\lot = z_1^n\bv_k + \lot \hbox{ whenever }l+(d-1)(m-1)=n.
\end{equation}
From the case $l=0$, $\calM(\bv_k)\subseteq\calM_{\bv_k}(z_1)$.  
Thus $\overline T(K,\calM_{\bv_k}(z_1))\leq T(K,\lambda_k)$.  
It remains to show that $T(K,\lambda_k)\leq \underline T(K,\calM_{\bv_k}(z_1))$.  

Let $\epsilon>0$.  By  Proposition \ref{prop:Tz1z2}, we can choose $N>d$ sufficiently large that 
$$T_n(K,\calM_{z_1^l}(\bv_k))\geq (1-\epsilon)T(K,\lambda_k), \hbox{ for all } n\geq N, \ l\in\{0,\ldots,d-2\}.$$  
Given $n\in \N$,  the division algorithm determines $m,l$ uniquely when $n=(d-1)(m-1)+l$ and $l<d-1$; then 
$$
(1-\epsilon)T(K,\lambda_k) \leq T_{n}(K,\calM_{z_1^l}(\bv_k)) = T_n(K,\calM_{\bv_k}(z_1))
$$
for all $n\geq N$ where equality of the last two quantities follows from \eqref{eqn:nml}.  Finally, take the liminf as $n\to\infty$ and let $\epsilon\to 0$. 
\end{proof}

We construct more polynomial classes using ordered bases.  Precisely, consider polynomials in $\text{span}(\calB)$, where $\calB=\{\bb_j\}_{j=1}^{\infty}$ is linearly independent and listed according to a  \emph{graded ordering}, i.e., $\deg(\bb_j)\leq\deg(\bb_{j+1})$ for all $j$.  

\begin{notation*}\rm 
Given a collection $\calB=\{\bb_j\}_{j=1}^{\infty}$ with a graded ordering, and a polynomial $p\in\text{span}(\calB)$,  $p=\sum_{k=1}^{j}c_k\bb_k$ with $c_j\neq 0$, write $\lot_{\calB}(p)$ to denote an arbitrary polynomial of the form $\sum_{k=1}^{j-1}a_k\bb_k$.  Also, $p+\lot_{\calB}:=p+\lot_{\calB}(p)$.  
\end{notation*}

Recall that the basis $\calC$ is ordered so that $z_1^l\bv_j^n$ comes before  $z_1^l\bv_{j+1}^n$ for each $j$.  Define the classes
$$
\widetilde\calM_l(\bv_j) :=\{p\in \C[V]\colon  p= z_1^l\bv_j^n+\lot_{\calC},\ n\in \N\}, \quad \widetilde\calM(\bv_j) := \bigcup_{l=0}^{d-2}\widetilde\calM_l(\bv_j).
 $$

\begin{lemma}
We have $T(K, \lambda_j) = T(K,\widetilde\calM_l(\bv_j))$ for all $l\in\{0,\ldots,d-2\}$.  Hence $T(K,\lambda_j)=T(K,\widetilde\calM(\bv_j))$.  
\end{lemma}
\begin{proof}
Fix $l$.  Clearly $\calM_{z_1^l}(\bv_j)\subseteq \widetilde\calM_l(\bv_j)$, so $T(K, \calM_{z_1^l}(\bv_j)) \geq \overline{\underline T}(K,\widetilde\calM(\bv_j))$.

For the reverse inequality, if $p\in\widetilde\calM_l(\bv_j)$, then
$p=z_1^l\bv_j^n+\lot_{\calC}$ for some $n\in \N$ and  
$$
(z_1^l\bv_j^n+\lot_{\calC})\bv_j = (z_1^l\bv_j^n + \sum_{k\neq j}c_kz_1^l\bv_k^n + \lot)\bv_j = z_1^l\bv_j^{n+1} + \lot\in\calM_{z_1^l}(\bv_j).   
$$
Hence $p\mapsto p\bv_j$ maps $\widetilde\calM_l(\bv_j)$  to  $\calM_{z_1^l}(\bv_j)$. Applying Lemmas \ref{lem:Tinf}--\ref{lem:Tsup},  $T(K, \calM_{z_1^l}(\bv_j)) \leq \overline{\underline T}(K,\widetilde\calM(\bv_j))$. 

Thus $\overline{\underline T}(K,\widetilde\calM(\bv_j))=T(K, \calM_{z_1^l}(\bv_j))$, and by Proposition  \ref{prop:Tz1z2} we get the first statement.  The second statement follows from the first.
\end{proof}

Another basis of $ \C[A]$ is the monomial basis $\calS$ consisting of all monomials of degree $\leq d-2$, and $z_1^{n},  z_1^{n-1}z_2,\ldots,z_1^{n-d+1}z_2^{d-1}$ in degrees $n\geq d-1$.  We use the \emph{grevlex ordering}, where we order by increasing degree, and by increasing powers of $z_2$ within the same degree.  
For each $k\in\{0,\ldots,d-2\}$, define  $$\calZ(k) = \{ p\in \C[V]\colon p(z_1,z_2)=z_2^kz_1^n+\lot_{\calS} \}.$$

With this ordering, $\calZ(0) =\calM(z_1)$ so $T(K,\calZ(0))$ exists  by Proposition \ref{prop:classical}.  Applying Lemmas \ref{lem:Tinf}--\ref{lem:Tsup} to the map $p\mapsto z_2p$, we also have monotonicity:
$$\overline{\underline T}(K,\calZ(k-1))\geq  \overline{\underline T}(K,\calZ(k)) \hbox{ for all } k\geq 1.$$  

It turns out that $T(K,\calZ(k))$ exists for all $k\in\{0,\ldots,d-2\}$ and coincides with the $(k+1)$-st largest directional Chebyshev constant.  For the moment we show the following.

\begin{proposition}\label{prop:TZ}
Suppose the directions $\lambda_1,\ldots,\lambda_d$ are labelled so that 
$$T(K,\lambda_1)\geq T(K,\lambda_2) \geq\cdots\geq T(K,\lambda_d) .$$ Then  
$T(K,\calZ(0))=T(K,\lambda_1)$ and  $\overline{\underline T}(K,\calZ(k))\leq T(K,\lambda_{k+1})$ for each $k\in\{1,\ldots,d-1\}$.  
\end{proposition}

\begin{proof}
We first do the following computation in $ \C[A]$:
$$\begin{aligned}
\Pi^{(k)} z_1^{d-1}:=\left(\prod_{j=1}^{k} (z_2-\lambda_j z_1)\right)z_1^{d-1} &=
\left(\prod_{j=1}^{k} (z_2-\lambda_j z_1)\right) \left(\sum_{l=1}^d\bv_l +\lot \right) \\
 &= \sum_{l=k+1}^d \left(\prod_{j=1}^{k}(\lambda_l-\lambda_j)\right)z_1^{k}\bv_l +\lot \\
&:= \sum_{l=k+1}^d c_{kl}z_1^{k}\bv_l + \lot.
\end{aligned}$$
The equality in the second line uses Corollary \ref{polyprop}. 
Now $z_2^kz_1^n+\lot_{\calS} = \Pi^{(k)}z_1^n+\lot_{\calS}$ by expanding $\Pi^{(k)}$; using this and the previous calculation,
\begin{equation}\label{eqn:ap}
\begin{aligned}\calZ(k)&=\left\{ \Pi^{(k)}z_1^n+\lot_{\calS}\right\} \\ 
&= \left\{ \sum_{l=k+1}^d c_{kl}z_1^{n+k-d+1}\bv_l + \lot_{\calS} \right\}  \\
&\supseteq \left\{ \sum_{l=k+1}^d c_{kl}z_1^{n+k-d+1}\bv_l + \lot \right\} = \calM_{\sum_{k+1}^d c_{kl}\bv_l}(z_1).  \end{aligned}\end{equation} 
Using Corollary \ref{cor:3}, Lemma \ref{lem:Tsum} and definition (\ref{dirchebc}), 
$$\overline{\underline T}(K,\calZ(k))\leq T(K,\calM_{\sum_{k+1}^d c_{kl}\bv_l}(z_1))   \leq \max_{l\geq k+1}\{ T(K,\lambda_l)\} = T(K,\lambda_{k+1}).$$ 
When $k=0$ we also have  
$$
T(K,\lambda_1)=T(K,\calM_{\bv_1}(z_1))\leq T(K,\calM(z_1)) = T(K,\calZ(0))
$$
from Corollary \ref{cor:oldT}. Thus $T(K,\calZ(0)) =T(K,\lambda_1)$.
\end{proof}

We recall the construction of transfinite diameter of a compact set $K$ with respect 
to a collection of polynomials $\calB=\{\bb_j\}_{j=1}^{\infty}$ arranged in a graded ordering.  

Given $\{\zeta_1,\ldots,\zeta_n\}\subset K$ define $\VDM_{\calB}(\zeta_1,\ldots,\zeta_n):=\det\bigl[\bb_j(\zeta_k)\bigr]_{j,k=1}^n$ and 
$$
V_n=V_n(K,\calB):=\sup\{|\VDM_{\calB}(\zeta_1,\ldots,\zeta_n)|\colon \zeta_j\in K,\ \forall j\}. 
$$
For each $n\in \N$, let $m_n$ be the number of polynomials in $\calB$ of degree at most $n$ and let $l_n:=\sum_{j=1}^{m_n} \deg(\bb_j)$ be the sum of the degrees. The \emph{transfinite diameter of $K$ with respect to $\calB$} is given by 
$$
d_{\calB}(K):= \limsup_{n\to\infty} \left(V_{m_n}\right)^{1/l_n}.
$$

Using $\calB$ we also define Chebyshev constants 
$$
\tau_n=\tau_n(K,\calB):= \inf\{\|p\|_K^{1/\deg(p)}\colon p\in\text{span}(\calB),\ p(z) = \bb_n(z)+\lot_{\calB}\}.
$$
The inequality 
\begin{equation}\label{eqn:vn<tn}
\frac{V_{n}}{V_{n-1}} \leq n\tau_n^{deg(\bb_n)} \hbox{ for all }n\in \N  
\end{equation}
is well-known (cf., \cite{Z}) and will be used in what follows.

For the rest of this section $\{\bb_j\}_{j=1}^{\infty}$  will be the (grevlex ordered) monomial basis $\calS$.  Let $N_0\in\N$ be such that $z_1^{d-1}=\bb_{N_0}$; then 
$z_1^{d-1+m-k}z_2^k = \bb_{N_0+md+k}$  for each  $m\in\N$.  Instead of $m$, we can relabel these monomials using the parameter $n:=d-1+m$, which is their total degree.  Setting $M:=N_0-d^2+d$, a calculation yields
$$
z_1^{n-k}z_2^k = \bb_{M+nd+k}, \quad  n\in\N, \  n\geq d.
$$
It follows that $\tau_{M+nd+k}^n=  T_n(K,\calZ(k))^n$, so 
\begin{equation}\label{eqn:bartau}
\limsup_{n\to\infty} \tau_{M+k+nd} = \overline T(K,\calZ(k))=:\overline\tau(M+k+nd).\end{equation}
(The notation on the right will be useful shortly; it associates the lim sup with the indices of those monomials of the form $z_1^{n-k}z_2^k$.)

\bigskip

In fact, the quantities  $T(K,\calZ(k))$ exist ($\limsup$ in \eqref{eqn:bartau} may be replaced by $\lim$) and we have a formula for the transfinite diameter in terms of these quantities.

\begin{theorem}\label{thm:ds=dc}
Let $K\subset A$ be compact, and suppose the directions $\lambda_1,\ldots,\lambda_d$ are labelled so that $T(K,\lambda_k)\geq T(K,\lambda_{k+1})$ for all $k\in\{1,\ldots, d-2\}$.  Then 
\begin{equation}\label{eqn:ds=dc}
\left(\prod_{k=1}^d  T(K,\calZ(k-1))\right)^{1/d} =  d_{\calS}(K) = d_{\calC}(K) = \left(\prod_{k=1}^d T(K,\lambda_k)\right)^{1/d},
\end{equation}
and hence $T(K,\calZ(k-1))=T(K,\lambda_k)$ for all $k\in\{1,\ldots,d\}$.
\end{theorem}

We first prove the following.

\begin{proposition} \label{414}
Theorem \ref{thm:ds=dc} holds with $\overline T(K,\calZ(k))$ in place of $T(K,\calZ(k))$ for each $k$.  
\end{proposition}

\begin{proof}
The second and third equalities of \eqref{eqn:ds=dc} have been proved in \cite{M}, Theorem 5.7 and Corollary 5.14.  Let us write the common quantity given by these expressions as $d(K)$.  

We will use Proposition \ref{prop:TZ} to show that  
\begin{equation}\label{eqn:sT=d} 
d(K) = \left(\prod_{k=1}^d \overline T(K,\calZ(k-1))\right)^{1/d}. \end{equation}

Let $\epsilon>0$.  Choose a large $n_0\in\N$ such that for all $k\in\{0,\ldots,d-1\}$,  
$$\tau_{M+k+nd}\leq (1+\epsilon)\stau(M+k+nd)  \hbox{  whenever }  n\geq n_0$$ (recall  \eqref{eqn:bartau} for the notation).  Let $N:=M+n_0$; then  
$$\stau(n)=T(K,\calZ(k)) \hbox{ whenever }n\geq N \hbox{ and }   n-N\equiv k \mod d.$$ 
For convenience  write $\alpha_j=\deg(\bb_j)$; then 
\begin{equation}\label{eqn:alpha_j} \alpha_j= \frac{j}{d} +O(1).  \end{equation}

Now let  $n\geq N$, and write $n=N+dm+k$ for some nonnegative integers $m,k$ with $k<d$.  Bound $V_{n}$ from above using \eqref{eqn:vn<tn} and a telescoping product.   We have 
\begin{equation}\label{eqn:tp} \begin{aligned}
V_n=V_{N-1}\frac{V_{N}}{V_{N-1}}\cdots\frac{V_n}{V_{n-1}}&\leq V_{N-1} \frac{n!}{N!} \prod_{\nu=0}^{n-N}  \tau_{\nu+N}^{\alpha_{\nu+N}} \\ 
&   \\
&\leq V_{N-1} \frac{n!}{N!} \prod_{\nu=0}^{dm} (1+\epsilon)^{\alpha_{\nu+N}}\stau(\nu+N)^{\alpha_{\nu+N}} 
 \prod_{\nu=dm+1}^{dm+k} C^{\alpha_{\nu+N}} \\ 
\end{aligned}\end{equation}
where we use some fixed constant $C\geq(1+\epsilon)\max\{\overline T(K,\calZ(0)),\ldots,\overline T(K,\calZ(d-1))  \}$  to estimate the last $k$ terms of the product.

Using \eqref{eqn:alpha_j}, a calculation gives
$$
\sum_{\nu=0}^{dm} \alpha_{\nu+N} = \frac{dm^2}{2} + O(m),\quad  
\sum_{\nu=dm+1}^{dm+k} \alpha_{\nu+N} = km+O(1) = O(m),
$$
 so we can estimate the upper bound in \eqref{eqn:tp} from above by   
\begin{equation}\label{eqn:<estimate}V_{N-1} \frac{(d(m+1)+N)!}{N!}C^{O(m)}(1+\epsilon)^{\frac{dm^2}{2} + O(m)} \prod_{\nu=0}^{dm} \stau(\nu+N)^{\alpha_{\nu+N}}.\end{equation}
Now $\stau(\nu+N) = \overline T(K,\calZ(k))$ whenever $\nu = k+ds$, so 
\begin{equation}\label{eqn:stauT}\begin{aligned}
\prod_{\nu=0}^{dm} \stau(\nu+N)^{\alpha_{\nu+N}} &= \prod_{k=1}^d\prod_{s=0}^m \overline T(K,\calZ(k-1))^{s+O(1)} \\
& = \prod_{k=1}^d \overline T(K,\calZ(k-1))^{\frac{m^2}{2}+O(1)}.
\end{aligned}\end{equation}

By definition $n=dm+O(1)$, so 
\begin{equation}\label{eqn:ln} l_n=\frac{dm^2}{2}+O(m).\end{equation}
  Taking $l_n$-th roots of $V_n$ and using \eqref{eqn:<estimate} and  \eqref{eqn:stauT}, we have (when $n\to\infty$,  equivalently $m\to\infty$) that 
$$
d(K) = \lim_{n\to\infty} (V_n)^{1/l_n}\leq (1+\epsilon) \prod_{k=1}^d\overline T(K,\calZ(k-1))^{1/d}, 
$$
where, in \eqref{eqn:<estimate}, the factors preceding the term $(1+\epsilon)^{\frac{dm^2}{2} + O(m)}$ all go to 1 in the limit upon taking roots.    (For the factorial terms, use the standard fact that $(m!)^{1/m^2}\to 1$ as $m\to\infty$.)  Letting $\epsilon\to 0$, we have $d(K)\leq \prod_{k=1}^d\overline T(K,\calZ(k-1))^{1/d}$.   On the other hand,
$$
\left(\prod_{k=1}^d\overline T(K,\calZ(k-1))\right)^{1/d}\leq \left(\prod_{k=1}^d T(K,\lambda_k)\right)^{1/d} =d(K)
$$
by Proposition \ref{prop:TZ}, and we get the reverse inequality.  
\end{proof}

We will now complete the proof of Theorem \ref{thm:ds=dc} by showing that the limits defining $T(K,\calZ(k))$ exist; equivalently,  
$$\underline T(K,\calZ(k))=\overline T(K,\calZ(k)) \hbox{ for each }k. $$
The proof for all cases is the same; we do it for $k=1$.   In this case, we are looking at Chebyshev constants associated to polynomials of the form $z_1^mz_2+\lot_{\calS}$.    
In what follows, redefine $N\in\N$ by $\bb_N=z_1^{d-2}z_2$; then the Chebyshev constants of interest are $\tau_{N+dm}$. We begin with the following.

\begin{proposition} \label{prop:13}
Let $\Sigma_m:=\sum_{\nu=1}^m \nu$.  Then
\begin{equation}\label{eqn:TZ1}
\overline T(K,\calZ(1)) = \lim_{m\to\infty} \left(\prod_{\nu=1}^m \tau_{N+d\nu}^{\nu}\right)^{1/\Sigma_m} .
\end{equation}
\end{proposition}

\begin{proof}
We estimate as in \eqref{eqn:tp}, with $n=N-2 +dm$, $m\in\N$. (We will then let $m\to\infty$.)  We have 
\begin{equation}\label{eqn:prop415a}
V_n \leq V_{N-2}\frac{n!}{(N-2)!}\prod_{j=N-1}^n \tau_j^{\alpha_j}  = V_{N-2}\frac{n!}{(N-2)!} \prod_{k=0}^{d-1}\prod_{\nu=0}^m \tau_{N-1+d\nu+k}^{d-1+\nu}.
\end{equation}
 As $m\to\infty$ we have  $n\to\infty$, and therefore 
\begin{equation}\label{eqn:T1a}
\begin{aligned}
d(K)  \leq   \limsup_{m\to\infty}\left(\prod_{k=0}^{d-1}\prod_{\nu=0}^m \tau_{N-1+d\nu+k}^{d-1+\nu} \right)^{1/l_n}. 
\end{aligned}\end{equation}
(As before, the other factors on the right-hand side of \eqref{eqn:prop415a} go to $1$ in the limit.)  We will use \eqref{eqn:T1a} to show that
\begin{equation}\label{eqn:T1b}
\lim_{m\to\infty}\left(\prod_{\nu=0}^m \tau_{N+d\nu}^{d-1+\nu}\right)^{1/l_n} = \overline T(K,\calZ(1))^{1/d}.
\end{equation}
Let $\epsilon>0$.  Then for each $k$, $\tau_{N-1+d\nu+k}<(1+\epsilon)\overline T(K,\calZ(k))$ holds for all but finitely many $\nu$; we apply this to all terms where $k\neq 1$ on the right-hand side of \eqref{eqn:T1a}: 
$$\begin{aligned}
\limsup_{m\to\infty}\left(\prod_{k=0}^{d-1}\prod_{\nu=0}^m \tau_{N-1+d\nu+k}^{d-1+\nu} \right)^{1/l_n}
  & \leq \limsup_{m\to\infty}\left(\prod_{k\neq 1}\prod_{\nu=0}^m (1+\epsilon)\overline T(K,\calZ(k))^{d-1+\nu} \right)^{1/l_n} \\
& \qquad \times \limsup_{m\to\infty}\left(\prod_{\nu=0}^m \tau_{N+d\nu}^{d-1+\nu} \right)^{1/l_n} \\
& = \left((1+\epsilon)\prod_{k\neq 1} \overline T(K,\calZ(k)))\right)^{1/d} \\
&\qquad \times \limsup_{m\to\infty}\left(\prod_{\nu=0}^m \tau_{N+d\nu}^{d-1+\nu} \right)^{1/l_n}. \\
\end{aligned}$$ 
Then using Theorem \ref{thm:ds=dc}, 
$$\begin{aligned}
\left(\prod_{k=0}^{d-1} \overline T(K,\calZ(k)))\right)^{1/d}  =d(K)
& \leq \left((1+\epsilon)\prod_{k\neq 1} \overline T(K,\calZ(k)))\right)^{1/d} \\ 
&\qquad \times \limsup_{m\to\infty}\left(\prod_{\nu=0}^m \tau_{N+d\nu}^{d-1+\nu} \right)^{1/l_n}. 
\end{aligned}$$  Cancelling the common factors $\overline T(K,\calZ(k))$ for $k\neq 1$ on each side, and letting $\epsilon\to 0$, 
$$
\overline T(K,\calZ(1)))^{1/d} \leq \limsup_{m\to\infty}\left(\prod_{\nu=0}^m \tau_{N+d\nu}^{d-1+\nu} \right)^{1/l_n}.
$$
To get  a reverse inequality, let $\epsilon>0$.  Then $\tau_{N+dm}<(1+\epsilon)\overline T(K,\calZ(1))$ for all but finitely many $m\in\N$, so  
$$\begin{aligned}
\limsup_{m\to\infty}\left(\prod_{\nu=0}^m \tau_{N+d\nu}^{d-1+\nu} \right)^{1/l_n} &\leq \limsup_{m\to\infty}  \left(\prod_{\nu=0}^m ((1+\epsilon)\overline T(K,\calZ(1)))^{d-1+\nu}  \right)^{1/l_n} \\
& =    \limsup_{m\to\infty}\left( (1+\epsilon)\overline T(K,\calZ(1)) \right)^{\left(\sum_{\nu=0}^{m}(d-1-\nu)\right)/l_n }.
\end{aligned}$$ 
Using \eqref{eqn:ln}, we have  $\frac{1}{l_n}\sum_{\nu=0}^{m}(d-1-\nu)= \frac{m^2+O(m)}{dm^2+O(m)}\to \frac{1}{ d}$  as $m\to\infty$,  
 and we obtain  $$\limsup\limits_{m\to\infty}\left(\prod_{\nu=0}^m \tau_{N+d\nu}^{d-1+\nu} \right)^{1/l_n}\leq \left((1+\epsilon)\overline T(K,\calZ(k))\right)^{1/d}.$$  Letting $\epsilon\to 0$ we get the reverse inequality; altogether, 
$$
\overline T(K,\calZ(1)))^{1/d} = \limsup_{m\to\infty}\left(\prod_{\nu=0}^m \tau_{N+d\nu}^{d-1+\nu} \right)^{1/l_n}.
$$
This is almost \eqref{eqn:T1b}, except that we need limsup replaced by lim.  

For the sake of obtaining a contradiction, suppose the limsup is not a limit. Then there exists $\delta>0$ and 
a subsequence $\{m_j\}$ in the parameter $m$ such that 
$$
\overline T(K,\calZ(1)))^{1/d}-\delta = \lim_{k\to\infty} \left(\prod_{\nu=0}^{m_j} \tau_{N+d\nu}^{d-1+\nu} \right)^{1/l_{n_j}}.
$$
But then, by similar reasoning as above with the limit formula $d(K)=\lim_{j\to\infty} (V_{n_j})^{1/l_{n_j}}$ (where $n_j=N-2+dm_j$),  we obtain for any $\epsilon>0$, 
$$\begin{aligned}
\left(\prod_{k=0}^{d-1} \overline T(K,\calZ(k))\right)^{1/d}  
 \leq \left((1+\epsilon)\prod_{k\neq 1} \overline T(K,\calZ(k)))\right)^{1/d}  
\left(\overline T(K,\calZ(1)))^{1/d}-\delta\right)
\end{aligned}$$
and therefore $\overline T(K,\calZ(1))^{\frac{1}{d}}\leq (1+\epsilon)^{1-\frac{1}{d}}(\overline T(K,\calZ(1)))^{\frac{1}{d}}-\delta )$.  
Letting $\epsilon\to 0$ gives a contradiction.  Thus \eqref{eqn:T1b} holds.  

To finish the proof, we must modify \eqref{eqn:T1b} to obtain \eqref{eqn:TZ1}.  First, we may replace  $\tau_{N+d\nu}^{d-1+\nu}$ in the product by  $\tau_{N+d\nu}^{\nu}$. This amounts to dividing out $\tau_{N+d\nu}^{d-1}$ for each $\nu=1,\ldots,m$.  If we estimate all of these $\tau_{N+d\nu}$ quantities by a uniform constant $C$ then  \eqref{eqn:T1b} is off by an estimated $C^{m(d-1)}$  relative to \eqref{eqn:TZ1}; the $l_n$-th root of this goes to 1 as $m\to\infty$  (recall $l_n=dm^2/2+O(m)$).  

Finally,  $\Sigma_m= \frac{m^2}{2}+O(m)$, so $\frac{l_n}{\Sigma_m}\to d$.  Replacing the $l_n$-th root in the formula by the $\Sigma_m$-th root,  the new limit as $m\to\infty$ is the $d$-th power of the original limit.  Hence \eqref{eqn:TZ1} holds.  
\end{proof}

For convenience, write $\overline{\underline T}:=\overline{\underline T}(K,\calZ(1))$ and $T_m:=\tau_{N+dm}$, and rewrite \eqref{eqn:TZ1} in the form \def\sT{\overline T}
\begin{equation}\label{eqn:logTZ1}
\log\sT = \lim_{m\to\infty}\left(\frac{1}{\Sigma_m}\sum_{\nu=1}^m \nu\log T_{\nu}\right).
\end{equation}
\def\bP{\mathbf{P}}
Define for $\delta,\theta>0$ and $m\in \N$ the statement
\begin{equation*}
\bP(m,\delta,\theta)\equiv \Bigl[ \log T_{m+s}<\log\sT -\delta \hbox{ for all } s\in[0,\theta m)\cap \Z \Bigr].  
\end{equation*}
We use Proposition \ref{prop:13} for the next result.

\begin{lemma}
Fix $\delta,\theta>0$.  Then $\bP(m,\delta,\theta)$ is true for at most finitely many $m\in \N$.  
\end{lemma}

\begin{proof}  We prove this by contradiction.

Suppose the conclusion is false, i.e., there are infinitely many values of $m\in \N$ where $\bP(m,\delta,\theta)$ is true.  Arrange them into a strictly increasing sequence $\{m_k\}_{k=1}^{\infty}$.  

For each $k$, let $s_k$ be the unique integer for which $\theta m_k-1<s_k\leq \theta m_k$; so 
$$s_k= m_k+O(1) \hbox{ and  } \Sigma_{m_k+s_k}=\frac{1}{2}(1+\theta)^2m_k^2+O(m_k).$$ 
 We have  
$$
\frac{1}{\Sigma_{m_k+s_k}}\sum_{\nu=1}^{m_k+s_k} \nu\log T_{\nu} = \frac{1}{\Sigma_{m_k+s_k}}\sum_{\nu=1}^{m_k} \nu\log T_{\nu}  + \frac{1}{\Sigma_{m_k+s_k}}\sum_{\nu=m_k+1}^{m_k+s_k} \nu\log T_{\nu}.
$$
For convenience, write this sum as $S(k)=S_1(k)+S_2(k)$.  

We have $S(k)\to\log\sT$ as $k\to\infty$ by Proposition \ref{prop:13}.  Next,  
$$
S_1(k) = \frac{\Sigma_{m_k}}{\Sigma_{m_k+s_k}}\left(\frac{1}{\Sigma_{m_k}}\sum_{\nu=1}^{m_k} \nu\log T_{\nu} \right) \longrightarrow \frac{1}{(1+\theta)^2}\log\sT \hbox{ as } k\to\infty,
$$
again by Proposition \ref{prop:13}.   For $S_2(k)$, we use $\bP(m_k,\delta,\theta)$ to estimate as follows:
$$\begin{aligned}
S_2(k) = \frac{1}{\Sigma_{m_k+s_k}}\sum_{\nu=m_k+1}^{m_k+s_k} \nu\log T_{\nu}
&< (\log\sT -\delta)\left(\frac{1}{\Sigma_{m_k+s_k}}\sum_{\nu=m_k+1}^{m_k+s_k} \nu \right) \\
&\longrightarrow (\log\sT -\delta)\left( 1-\frac{1}{(1+\theta)^2}\right) \hbox{ as }k\to\infty.
\end{aligned}$$
Finally, looking at the limit of $S(k)=S_1(k)+S_2(k)$ on both sides as $k\to\infty$, 
$$\begin{aligned}
\log(\sT) &\leq \frac{1}{(1+\theta)^2}\log\sT + (\log\sT -\delta)\left( 1-\frac{1}{(1+\theta)^2}\right) \\
& = \log\sT - \delta\left( 1-\frac{1}{(1+\theta)^2}\right) 
\end{aligned}$$
which is a contradiction.  The proof is complete. 
\end{proof}

\begin{corollary}\label{cor:15}
Let $\delta,\theta>0$, and let $\{m_k\}_{k=1}^{\infty}$ be an increasing sequence in $ \N$.  Then there exists $k_0\in \N$ such that  for each $k\geq k_0$, 
$$ \log T_{m_k+s_k}\geq \log \sT -\delta \hbox{ for some }s_k\in[0,\theta m_k]\cap \Z. $$
\end{corollary}

\begin{proof}
By the previous lemma, $\bP(m_k,\delta,\theta)$ is true for only finitely many $k$, i.e., there exists  $k_0\in \N$ sufficiently large that $\bP(m_k,\delta,\theta)$ is false for all $k\geq k_0$. The conclusion is simply the negation of $\bP(m_k,\delta,\theta)$. 
\end{proof}

\begin{proposition}\label{prop:infT=supT}
We have $\underline T=\overline T$.
\end{proposition}

\begin{proof} \def\uT{\underline T}
For the purpose of obtaining a contradiction, suppose $\uT<\sT$.  

Let $\{m_k\}_{k=1}^{\infty}$ be a strictly increasing sequence in $ \N$ such that $T_{m_k}\to\uT$, and fix $\epsilon\in(0,\sT-\uT)$.  Replacing $\{m_k\}$ with its tail, we may assume that 
\begin{equation}\label{eqn:supT=infT1}
\log T_{m_k}< \log\sT -\epsilon \hbox{ for all } k.
\end{equation}

Let $\delta,\theta>0$.  By Corollary \ref{cor:15} there exists $k_0\in \N$ and a sequence $\{s_k\}$ in $[0,\theta m_k]\cap \Z$ such that 
\begin{equation}\label{eqn:supT=infT2}
\log T_{m_k+s_k}\geq \log\sT -\delta \hbox{ for all } k\geq k_0.
\end{equation}

Let $t_{m_k}(z)=z_1^{m_k-1}z_2+\lot_{\calS}$ be a polynomial such that $\|t_{m_k}\|_K=T_{m_k}^{m_k}$.  Then $z_1^{s_k}t_{m_k}(z)=z_1^{m_k+s_k-1}z_2+\lot_{\calS}$, so 
$$
T_{m_k+s_k}^{m_k+s_k}\leq \|z_1^{s_k}t_{m_k}\|_K\leq \|z_1^{s_k}\|_K\|t_{m_k}\|_K\leq C^{s_k}T_{m_k}^{m_k},
$$
choosing some $C>\|z_1\|_K$.  Hence $T_{m_k+s_k}^{1+s_k/m_k}\leq C^{s_k/m_k}T_{m_k}$.  Taking logs and using \eqref{eqn:supT=infT1}, \eqref{eqn:supT=infT2}, we obtain 
$$
\left(1+\frac{s_k}{m_k}\right)(\log\sT -\delta)\leq  \frac{s_k}{m_k}\log C + \log\sT-\epsilon, \hbox{ for all }k\geq k_0.
$$
Taking the limsup on both sides as $k\to\infty$, we have 
\begin{equation}\label{eqn:infT=supT5}
(1+\theta_1)(\log\sT-\delta)\leq \theta_2\log C +\log\sT-\epsilon \hbox{ for some }\theta_1,\theta_2 \in[0,\theta].
\end{equation}
Here, $\theta_1,\theta_2$ are either the $\limsup$ or $\liminf$ of $\frac{s_k}{m_k}$ as $k\to\infty$, depending on whether the signs of  $\log\sT -\delta,\log C$ are positive or negative.

Finally, observe that $\epsilon$ does not depend on $\delta,\theta$.  We can let $\delta,\theta\to 0$; then $\theta_1,\theta_2\to 0$ also, and \eqref{eqn:infT=supT5} reduces to the contradiction $\log\sT\leq\log\sT-\epsilon$.  This completes the proof. 
\end{proof}

Note that Theorem \ref{chebthm}, that $T(K,\calZ(k-1))=T(K,\lambda_k)$ for all $k\in\{1,\ldots,d\}$, is the final part of Theorem \ref{thm:ds=dc}.

\section{Extremal-like functions} In this section, we define two families of extremal-like functions associated to a compact, nonpolar subset $K\subset A \subset \C^2$, denoted $V_{K}^{(k)}, k=1,...,d$ and $\tilde V_{K}^{(k)}, k=0,...,d-1$, and we show that 
$$V_K^*= \max[V_{K}^{(1)},...,V_{K}^{(d)}] =\max[\tilde V_{K}^{(0)},...,\tilde V_{K}^{(d-1)}]  \ \hbox{on} \ A^0 \setminus K.$$
We will utilize the Chebyshev polynomials associated to the family of Chebyshev constants $T(K, \calM_{z_1^j}(\bv_k)), \  j=0,...,d-2$ to construct $V_{K}^{(k)}$; and we use the Chebyshev polynomials associated to the Chebyshev constants $T(K,\lambda_k)$ to construct $\tilde V_{K}^{(k)}$. 
 
Recall we have the basis $\calC$ of $\C[A]$ given by 
$$\begin{aligned}
&1,z_1,z_2,\ldots, z_1^{d-2},\ldots, z_2^{d-2},    \\
&\bv_1,\ldots,\bv_d, z_1\bv_1,\ldots,z_1\bv_d,\ldots, z_1^{d-2}\bv_1,\ldots,z_1^{d-2}\bv_d,\\ 
&\bv_1^2,\ldots,\bv_d^2,\ldots, z_1^{d-2}\bv_1^2,\ldots,z_1^{d-2}\bv_d^2,\\
&\ldots \\
&\bv_1^k,\ldots,\bv_d^k,\ldots, z_1^{d-2}\bv_1^k,\ldots,z_1^{d-2}\bv_d^k,\ldots 
\end{aligned}$$
We have shown 
\begin{equation} \label{jstuff} T(K,\lambda_k)=T(K, \calM_{z_1^j}(\bv_k)), \ j=0,...,d-2; \ k=1,...,d \end{equation}
in Proposition \ref{prop:Tz1z2} where
$$\calM_{z_1^j}(\bv_k)) :=\{ z_1^j\bv_k^n+\lot\colon n\in \N \}.$$
Given $k\in \{1,...,d\}$, for each $n>d-2$, by the division algorithm, there exists a unique positive integer $l=l(n)$ and $j=j(n)\in \{0,1,...,d-2\}$ so that  $n=l(d-1)+j$. From (\ref{jstuff}), we have  
$$\lim_{n\to \infty} T_l(K,\calM_{z_1^j}(\bv_k))=T(K,\lambda_k).$$
We let 
$$t_{n}^{(k)}(z_1,z_2)=z_1^{j(n)}\bv_k^{l(n)} + \lot$$
be a Chebyshev polynomial for this class; i.e., $T_l(K,\calM_{z_1^j}(\bv_k)) =||t_{n}^{(k)}||_K^{1/n}$, and we define, for $(z_1,z_2)\in A$,  
\begin{equation}\label{vkk} V_{K}^{(k)}(z_1,z_2):=[\limsup_{n\to \infty} \frac{1}{n}\log \frac{| t_{n}^{(k)}(z_1,z_2)|}{||t_n^{(k)}||_K}]^* \in L(A) \end{equation} 
where the uppersemicontinuous regularization is taken over points in $A$. Using the results in the previous sections, we will show these functions have the property that 
\begin{equation} \label{notildes} \rho(V_K^{(k)},\lambda_m)\leq \rho_K(\lambda_m) \ \hbox{if} \ m\not = k; \ \rho(V_K^{(k)},\lambda_k)=\rho_K(\lambda_k).\end{equation} 
It then follows from Theorem \ref{b98} that 
$$\max [V_{K}^{(1)}(z_1,z_2),...,V_{K}^{(d)}(z_1,z_2)]= V_K^*(z_1,z_2)  \ \hbox{on} \ A^0 \setminus K.$$

Recall that we ordered the directions $\lambda_j$ so that
$$T(K,\lambda_1) \geq T(K,\lambda_2) \geq \cdots \geq T(K,\lambda_d);$$
and since 
\begin{equation}\label{keyeq2} \rho_K(\lambda_k)= -\log T(K,\lambda_k), \ \hbox{we have} \ 
\rho_K(\lambda_1) \leq \rho_K(\lambda_2) \leq \cdots \leq \rho_K(\lambda_d)  \end{equation} 
(recall Corollary \ref{cor:oldT}).

From Theorem \ref{chebthm}, we also have 
\begin{equation}\label{keyeq} T(K,\calZ(k))=T(K,\lambda_{k+1}), \ k=0,...,d-2. \end{equation} 
Here, recall the monomial basis $\calS$ consists of all monomials of degree $\leq d-2$, and $z_1^{n},  z_1^{n-1}z_2,\ldots,z_1^{n-d+1}z_2^{d-1}$ in degrees $n\geq d-1$.  We use the \emph{grevlex ordering}, where we order degree by degree,  and by increasing powers of $z_2$ within the same degree. For each $k\in\{0,\ldots,d-1\}$, we then defined 
$$\calZ(k) = \{ p\in\C[A]\colon p(z_1,z_2)=z_2^kz_1^n+\lot_{\calS} \}.$$
Let 
$$\tilde t_{n}^{(k)}(z_1,z_2)=z_2^kz_1^n+\lot_{\calS}$$
be a Chebyshev polynomial for this class; i.e., $T_{n+k}(K,\calZ(k))= ||\tilde t_{n}^{(k)}||_K^{1/n+k}$. We define, for $(z_1,z_2)\in A$,   
\begin{equation}\label{vkkt} \tilde V_{K}^{(k)}(z_1,z_2):=[\limsup_{n\to \infty} \frac{1}{n+k}\log \frac{| \tilde t_{n}^{(k)}(z_1,z_2)|}{||\tilde t_{n}^{(k)}||_K}]^* \in L(A).\end{equation}
Note that $\tilde t_{n}^{(k)}$ may not be unique 
but for any choice we will see that these functions $\tilde V_{K}^{(k)}$ have the property that 

\begin{eqnarray}\label{tildes}
\begin{array}{l} \rho(\tilde V_{K}^{(k)},\lambda_j) \leq \rho_K(\lambda_{j}) \ \hbox{if } j \leq k;\\
 \rho(\tilde V_{K}^{(k)},\lambda_j)=  \rho_K(\lambda_{k+1}) \ \hbox{if } j= k+1; \\
 \rho_K(\lambda_{k+1})  \leq   \rho(\tilde V_{K}^{(k)},\lambda_j) \leq \rho_K(\lambda_{j}) \ \hbox{if } j \geq k+2
     \end{array} 
\end{eqnarray}
It then follows from Theorem \ref{b98} that 
$$\max [\tilde V_{K}^{(0)}(z_1,z_2),...,\tilde V_{K}^{(d-1)}(z_1,z_2)]= V_K^*(z_1,z_2)  \ \hbox{on} \ A^0 \setminus K.$$

In order to prove (\ref{notildes}) and (\ref{tildes}), we consider the Robin constants of these extremal-like functions. Recall that for $u\in L(A)$, 
$$\rho_u(\lambda_k)=\limsup_{|z_1|\to \infty, \ (z_1,z_2)\in A, \ z_2/z_1\to \lambda_k} [u(z_1,z_2)-\log |z_1|].$$
If $u=\frac{1}{n}\log |p|$ where $p$ is a polynomial of degree $n$ in $\C[A]$, then
$$\rho_u(\lambda_k)=\limsup_{|z_1|\to \infty, \ (z_1,z_2)\in A, \ z_2/z_1\to \lambda_k} [\frac{1}{n}\log |p(z_1,z_2)|-\log |z_1|]$$
$$=\limsup_{|z_1|\to \infty, \ (z_1,z_2)\in A, \ z_2/z_1\to \lambda_k} [\frac{1}{n}\log |\hat p(z_1,z_2)|-\log |z_1|]$$
where $\hat p$ is the top degree ($n$) homogeneous piece of $p$. Thus 
$$\hat p(z_1,z_2)= a_nz_1^n +a_{n-1}z_1^{n-1}z_2 +\cdots + a_{n-(d-2)}z_1^{n-(d-2)}z_2^{d-2}.$$

Hence,
$$\frac{1}{n}\log |\hat p(z_1,z_2)|-\log |z_1| = \log|a_n +a_{n-1}(z_2/z_1) + \cdots + a_{n-(d-2)}(z_2/z_1)^{d-2}|$$ and, as in (\ref{robpoly}), 
$$\rho_u(\lambda_k)=\limsup_{|z_1|\to \infty, \ z_2/z_1\to \lambda_k} \frac{1}{n}\log|a_n +a_{n-1}(z_2/z_1) + \cdots + a_{n-(d-2)}(z_2/z_1)^{d-2}|$$
$$=\frac{1}{n}\log|a_n +a_{n-1}\lambda_k + \cdots + a_{n-(d-2)}\lambda_k^{d-2}|=\frac{1}{n}\log|\hat p(1,\lambda_k)|.$$
Thus, for
$$V_K^{(k)}(z_1,z_2):=[\limsup_{n\to \infty} \frac{1}{n}\log \frac{| t_n^{(k)}(z_1,z_2)|}{||t_n^{(k)}||_K}]^*,$$
and we have that
$$\rho(V_K^{(k)},\lambda_k)\geq \limsup_{n\to \infty} \frac{1}{n}\log \frac{| \hat t_n^{(k)}(1,\lambda_k)|}{||t_n^{(k)}||_K}$$
$$=\limsup_{n\to \infty} \frac{1}{n}\log \frac{| 1^j \bv_k(1,\lambda_k)^l|}{||t_n^{(k)}||_K}$$
$$=\lim_{n\to \infty} \frac{1}{n}\log \frac{1}{||t_n^{(k)}||_K}=-\log T(K,\lambda_k)= \rho_K(\lambda_k)$$
from (\ref{keyeq2}). On the other hand, since $V_K^{(k)}(z_1,z_2)\leq V_K^*(z_1,z_2)$, 
for all $m,k=1,...,d$
$$\rho_{V_{K}^{(k)}}(\lambda_{m})\leq \rho_K(\lambda_{m}).$$ 
This verifies (\ref{notildes}).

Concerning (\ref{tildes}) and the functions $\tilde V_{K}^{(k)}$, we first observe that
since we know for each $k=0,...,d-1$, 
$$\tilde V_{K}^{(k)}(z_1,z_2)\leq V_K^*(z_1,z_2),$$
for all $j=1,...,d$, and all $k=0,...,d-1$, 
$$\rho_{\tilde V_{K}^{(k)}}(\lambda_{j})\leq \rho_K(\lambda_{j}).$$
Next, we have 
\begin{equation}\label{hateqn} \hat {\tilde t_{n}^{(k)}}(z_1,z_2)=z_2^k z_1^n+a_{n,k-1}z_2^{k-1}z_1^{n+1} +\cdots +a_{n,0}z_1^{n+k}.\end{equation} 
In particular, 
$\hat {\tilde t_{n}^{(0)}}(z_1,z_2)=z_1^n$ so that 
$\hat {\tilde t_{n}^{(0)}}(1,\lambda_j)=1$. 
Then
$$\rho_{\tilde V_{K}^{(0)}}(\lambda_{j})\geq \limsup_{n\to \infty} \frac{1}{n}\log \frac{|\hat {\tilde t_{n}^{(0)}}(1,\lambda_j)|}{||\tilde t_{n}^{(0)}||_K}=\limsup_{n\to \infty} \frac{1}{n}\log \frac{1}{||\tilde t_{n}^{(0)}||_K}$$
$$=-\log T(K,\calZ(0))=-\log T(K,\lambda_{1})= \rho_K(\lambda_{1})$$
from (\ref{keyeq}) and (\ref{keyeq2}) for all $j=1,...,d$. Thus 
$$\rho_{\tilde V_{K}^{(0)}}(\lambda_{1})=\rho_K(\lambda_{1}) \ \hbox{and}  \ \rho_K(\lambda_{1})\leq \rho_{\tilde V_{K}^{(0)}}(\lambda_{j})\leq  \rho_K(\lambda_{j}), \ j=2,...,d.$$ 

Now, for $k=1$,
$$\hat {\tilde t_{n}^{(1)}}(z_1,z_2)=z_2z_1^n+a_nz_1^{n+1}=z_1^n(z_2+a_nz_1)$$
so that for $j=1,...,d$,
$$\hat {\tilde t_{n}^{(1)}}(1,\lambda_j)= \lambda_j + a_n.$$
Thus
$$\rho_{\tilde V_{K}^{(1)}}(\lambda_j)\geq \limsup_{n\to \infty} \frac{1}{n+1}\log \frac{| \hat {\tilde t_{n}^{(1)}}(1,\lambda_j)|}{||\tilde t_{n}^{(1)}||_K}=\limsup_{n\to \infty} \frac{1}{n+1}\log \frac{| \lambda_j + a_n|}{||\tilde t_{n}^{(1)}||_K}.$$
For $j=1,...,d$, we recall that since $\tilde V_{K}^{(1)}\leq V_K^*$ we must have 
$$\rho_{\tilde V_{K}^{(1)}}(\lambda_j)\leq  \rho_K(\lambda_{j}).$$
In particular, $\rho_{\tilde V_{K}^{(1)}}(\lambda_1)\leq  \rho_K(\lambda_{1})$. Since $T(K,\calZ(1))=T(K,\lambda_{2})$ from (\ref{keyeq}), 
$$\lim_{n\to \infty} \frac{1}{n+1}\log \frac{1}{||\tilde t_{n}^{(1)}||_K}=-\log T(K,\lambda_{2}) =\rho_K(\lambda_2).$$
Hence we must have  
$$\limsup_{n\to \infty} \frac{1}{n+1}\log |\lambda_1 + a_n| \leq  \rho_K(\lambda_{1})- \rho_K(\lambda_{2})\leq 0.$$
Thus if $\rho_K(\lambda_{1}) < \rho_K(\lambda_{2})$, then $\lim_{n\to \infty}a_n = -\lambda_1$. Hence for all $j\not = 1$, since the $\lambda_j$ are distinct,
$$\limsup_{n\to \infty}  \frac{1}{n+1}\log |\lambda_j + a_n|= 0$$
and we have 
$$\rho_{\tilde V_{K}^{(1)}}(\lambda_2)= \cdots = \rho_{\tilde V_{K}^{(k)}}(\lambda_{d})=\rho_K(\lambda_{2}).$$

In general, for $k\in \{2,...,d-1\}$, from (\ref{hateqn}), 
$$\hat {\tilde t_{n}^{(k)}}(1,\lambda_j) =\lambda_j^k +a_{n,k-1}\lambda_j^{k-1} +\cdots +a_{n,0}.$$
Then
$$\rho_{\tilde V_{K}^{(k)}}(\lambda_j)\geq \limsup_{n\to \infty} \frac{1}{n+k}\log \frac{ |\hat {\tilde t_{n}^{(k)}}(1,\lambda_j)|}{||\tilde t_{n}^{(k)}||_K}=\limsup_{n\to \infty} \frac{1}{n+k}\log \frac{| \lambda_j^k +a_{n,k-1}\lambda_j^{k-1} +\cdots +a_{n,0}|}{||\tilde t_{n}^{(k)}||_K}.$$
By similar reasoning, under the assumption that $\rho_K(\lambda_{k}) < \rho_K(\lambda_{k+1})$,
using the fact that
$$\limsup_{n\to \infty} \frac{1}{n+k}\log \frac{1}{||\tilde t_{n}^{(k)}||_K}= -\log T(K,\lambda_{k+1})= \rho_K(\lambda_{k+1}),$$
we have 
$$\limsup_{n\to \infty}  \frac{1}{n+k} \log |\hat {\tilde t_{n}^{(k)}}(1,\lambda_j)|\leq \rho_K(\lambda_j)-\rho_K(\lambda_{k+1}) <0$$
for $j=1,...,k$ (recall (\ref{keyeq2})). Writing
$$\hat {\tilde t_{n}^{(k)}}(1,\lambda)=\lambda^k +a_{n,k-1}\lambda^{k-1} +\cdots +a_{n,0}=\prod_{j=1}^k(\lambda - r_{j,n}),$$
we have
$$\limsup_{n\to \infty}  \frac{1}{n+k} \log |\hat {\tilde t_{n}^{(k)}}(1,\lambda_j)|=\limsup_{n\to \infty}  \frac{1}{n+k} \log |\prod_{j=1}^k(\lambda_j - r_{j,n})|<0$$
for $j=1,...,k$. Thus, after possibly reordering the $k$ roots $r_{1,n},...,r_{k,n}$, we first choose a subsequence $\{n(1)\}$ of $\N$ so that $\{r_{1,n(1)}\}$ converges to $\lambda_1$; then we take a subsequence $\{n(2)\}$ of $\{n(1)\}$ so that $\{r_{2,n(2)}\}$ converges to $\lambda_2$; etc.; finally we take a subsequence $\{n(k)\}$ of $\{n(k-1)\}$ so that $\{r_{k,n(k)}\}$ converges to $\lambda_k$. For this subsequence $\{n(k)\}$, we have $\{r_{j,n(k)}\}$ converges to $\lambda_j$ for $j=1,...,k$. In particular, for $j=k+1,...,d$ we have 
$$\limsup_{n(k)\to \infty}  \frac{1}{n(k)+1} \log |\hat {\tilde t_{n(k)}^{(k)}}(1,\lambda_j)|=0.$$
We conclude that
$$\limsup_{n\to \infty}  \frac{1}{n+k} \log |\hat {\tilde t_{n}^{(k)}}(1,\lambda_j)|\geq 0$$
for $j=k+1,...,d$. Since $\rho_{\tilde V_{K}^{(k)}}(\lambda_{k+1}) \leq  \rho_K(\lambda_{k+1})$, it follows that 
$$\limsup_{n\to \infty}  \frac{1}{n+k} \log |\hat {\tilde t_{n}^{(k)}}(1,\lambda_{k+1})|= 0$$
so that, in fact, $\rho_{\tilde V_{K}^{(k)}}(\lambda_{k+1}) =  \rho_K(\lambda_{k+1})$. For $j=k+2,...,d$, we conclude that
$$\rho_K(\lambda_{k+1})=\rho_{\tilde V_{K}^{(k)}}(\lambda_{k+1})\leq \rho_{\tilde V_{K}^{(k)}}(\lambda_{j}) \leq \rho_K(\lambda_{j}).$$
This verifies (\ref{tildes}).

We have proved the following.

\begin{theorem} \label{vkkvkkt} For $K\subset A\subset \C^2$ nonpolar, we have
$$\max [V_{K}^{(1)}(z_1,z_2),...,V_{K}^{(d)}(z_1,z_2)]= V_K^*(z_1,z_2)  \ \hbox{on} \ A^0 \setminus K$$
where $V_K^{(k)}$ is defined in (\ref{vkk}). Furthermore, if  
$$\rho_K(\lambda_1) < \rho_K(\lambda_2) < \cdots < \rho_K(\lambda_d),$$
we have
$$\max [\tilde V_{K}^{(0)}(z_1,z_2),...,\tilde V_{K}^{(d-1)}(z_1,z_2)]= V_K^*(z_1,z_2)  \ \hbox{on} \ A^0 \setminus K$$
where $\tilde V_{K}^{(k)}$ is defined in (\ref{vkkt}).

\end{theorem}

We next verify certain properties of our extremal-like functions $V_K^{(k)}, \tilde V_K^{(k)}$. These are in $L(A)$, in particular, they are subharmonic on $A$. We assume in the following lemmas that $\epsilon >0$ is given and $R=R(\epsilon)$ is chosen so that the conditions of Proposition \ref{hmstuff} are satisfied. Thus we have pairwise disjoint domains $D_j, \ j=1,...,d$ in $A$ with $A\setminus \{z\in \C^N: |z|< R\}=D_1\cup \cdots \cup D_d$; $D_j$ is $\epsilon$ close to the linear asymptote $L_j$ of $A$; and the projection $\pi_j: D_j \to \C$ given by $\pi_j(z_1,z_2)=z_1$ is one-to-one. 

\begin{lemma}
For $k=1,...,d$, we have 
\begin{equation}\label{Vk=V}
\tilde V_K^{(k-1)}(z_1,z_2)=V_K^{(k)}(z_1,z_2) = V_K(z_1,z_2) \hbox{ for all } z\in D_k.
\end{equation}
\end{lemma}

\begin{proof} We prove the equality $V_K^{(k)} = V_K$ on $D_k$; the proof that $\tilde V_K^{(k-1)} = V_K$ on $D_k$ is identical. 
Let $w=(w_0,w_2)$ be local coordinates on $D_k\subset\C^2\subset\C\PP^2$  with $z:=z(w)$ given by the correspondence $[1:z_1:z_2]=[w_0:1:w_2]$. Define 
$W(w) := V_K^{(k)}(z(w))-V_K(z(w))$ if $w\neq (0,\lambda_k)$.  Then $W$ is subharmonic and $W\leq 0$ on $D_k$. Moreover, $W$ extends across $(0,\lambda_k)$ as a subharmonic function with 
$W(0,\lambda_k)=\rho_{V_K^{(k)}}(\lambda_k)-\rho_{K}(\lambda_k)= 0$.  The point $(0,\lambda_k)$ is an interior point of the extended domain, so by the maximum principle $W\equiv 0$ in $D_k$.
\end{proof}

\begin{lemma} \label{kjequal}
Let $\Omega$ be a connected component of $A\setminus K$ and suppose $\Omega\supset (D_j\cup D_k)$.  Then $V_K^{(j)}=V_K^{(k)}=\tilde V_K^{(j-1)}=\tilde V_K^{(k-1)}=V_K$ on $\Omega$.  
\end{lemma}

\begin{proof} We show $V_K^{(j)}=V_K^{(k)}=V_K$ on $\Omega$; the proof that $\tilde V_K^{(j-1)}=\tilde V_K^{(k-1)}=V_K$ on $\Omega$ is identical. 
Since $\Omega$ is open it is path connected.  Let $\gamma\colon[0,1]\to\Omega$ be a continuous path with $\gamma(0)\in D_j$ and $\gamma(1)\in D_k$ and let $$T:=\sup\{t\in[0,1]\colon V_K(\gamma(s))=V_K^{(j)}(\gamma(s)) \hbox{ for all } s\in[0,t]\}.$$ 
We want to show that $T=1$.

By the previous lemma, $T\geq 0$. Suppose $T<1$. Let $U$ be a neighborhood of $\gamma(T)$.  Then there exists $s\leq T$ such that $\gamma(s)\in U$ and $V_K(\gamma(s))=V_K^{(j)}(\gamma(s))$.  Using the same argument as in the previous lemma with the point $(0,\lambda_k)$ replaced by the point  $\gamma(s)\in U$, we conclude that $V_K^{(j)}\equiv V_K$ in $U$. In particular, $V_K(\gamma(T+\delta))=V_K^{(j)}(\gamma(T+\delta))$ for all sufficiently small $\delta>0$.  This contradicts the definition of $T$. 

 Thus $T=1$, hence $V_K^{(j)}=V_K$  on $\gamma([0,1])$.  Similarly, $V_K^{(k)}=V_K$ on $\gamma([0,1])$.  Since $\gamma$ was an arbitrary path connecting a point in $D_j$ with one in $D_k$, the result follows.  
\end{proof}

We conclude with some examples.

\begin{example}
Let $A=\{(z_1,z_2)\in\C^2\colon z_1^2-z_2^2=1\}$ and take $K:=\{(z_1,z_2)\in A\colon z_2\in[-1,1]\}$. Then $V_K(z_1,z_2)=\log|h(z_2)|$ where $h(\zeta):=\zeta+\sqrt{\zeta^2-1}$ is the inverse Joukowski map, for $v(z_1,z_2):=\log|h(z_2)||_A$ is in $L^+(A)$; $v=0$ on $K$; and $dd^cv=0$ on $A\setminus K$. Since $A\setminus K$ is connected, from Lemma \ref{kjequal} we must have $$V_K^{(1)}(z_1,z_2)=V_K^{(2)}(z_1,z_2)=\tilde V_K^{(0)}(z_1,z_2)=\tilde V_K^{(1)}(z_1,z_2)=\log|h(z_2)|$$ on $A\setminus K$. 
\end{example}

\begin{example} \label{twotwo} We again let $A=\{(z_1,z_2)\in \C^2: z_1^2 -z_2^2 =1\}$. The associated homogeneous variety $A_h=\{(z_1,z_2)\in \C^2: z_1^2 -z_2^2 =(z_1 -z_2)(z_1+z_2)=0\}$ is the union of two complex lines. Then $\bv_1= \frac{1}{2}(z_1-z_2)$ and $\bv_2= \frac{1}{2}(z_1+z_2)$. Take 
$K:=\{(z_1,z_2)\in A: |\bv_1|=|\bv_2|=1/2\}$. Since the basis $\calC$ for $\C[A]$ is 
$$1,\bv_1, \bv_2,\bv_1^2, \bv_2^2,...,$$
it is easy to see that $t_n^{(1)}(z_1,z_2)= \bv_1^n$ and $V_{K}^{(1)}(z_1,z_2)=\log |z_1-z_2|$ while $t_n^{(2)}(z_1,z_2)= \bv_2^n$ and $V_{K}^{(2)}(z_1,z_2)=\log |z_1+z_2|$. Here, the Robin constants $\rho_K(\lambda_1), \rho_K(\lambda_2)$ are equal and since Theorem \ref{vkkvkkt} gives
$$V_K(z_1,z_2)= \max[V_{K}^{(1)}(z_1,z_2), V_{K}^{(2)}(z_1,z_2)] \ \hbox{on} \ A^0 \setminus K,$$
we have
\begin{equation}\label{first} V_K(z_1,z_2)=\max[\log^+|z_1-z_2|,\log^+|z_1+z_2|] \ \hbox{on} \ A. \end{equation}
Note that setting $u:=z_1-z_2$ and $v:=z_1+z_2$, we have $A = \{(u,v)\in \C^2: uv = 1\}$ and $K=\{(u,v)=(e^{it}, e^{-it})\in A: t\in [0,2\pi]\}$. 
It was shown in \cite{BL} that 
\begin{equation}\label{second} V_K(u,v)=\max[ \log^+|u|, \log^+|v|]\ \hbox{for} \ (u,v)\in A. \end{equation}
 Setting $u:=z_1-z_2$ and $v:=z_1+z_2$ in (\ref{first}) recovers (\ref{second}).

\end{example}

\begin{example}  For $\epsilon >0$, let $A_{\epsilon}:=\{(z_1,z_2)\in \C^2: z_1 z_2 = \epsilon\}$ and $K_{\epsilon}:=A_{\epsilon}\cap B$ where $B=\{(z_1,z_2)\in \C^2: |z_1|\leq 1, \ |z_2|\leq  1\}$ is the unit bidisk. We claim that 
$$V_{K_{\epsilon}}(z_1,z_2) = \max[ \log^+|z_1|, \log^+|z_2|] \ \hbox{on}  \ A_{\epsilon}.$$ To see this, simply note that $v(z_1,z_2):=\max[ \log^+|z_1|, \log^+|z_2|]|_{A_{\epsilon}}\in L^+(A_{\epsilon})$; $v=0$ on $K_{\epsilon}$; and $dd^cv=0$ on $A_{\epsilon}\setminus K_{\epsilon}$. 

The coordinate axes $z_1=0$ and $z_2=0$ are linear asymptotes for $A_{\epsilon}$. Thus this example does not satisfy the italicized conditions in the introduction. However, we can take the basis
$$1,z_1,z_2,z_1^2,z_2^2,...,z_1^n,z_2^n,...$$
given by $\calS$ for $\C[A_{\epsilon}]$ to compute the Chebyshev polynomials $\tilde t_n^{(j)}$ for $j=0,1$. Since the map $t\to (t, \epsilon/t)$ from the annulus $\{t\in \C: \epsilon \leq |t| \leq 1\}$ onto $K_{\epsilon}$ is holomorphic, one can check that we have $\tilde t_n^{(0)}(z_1,z_2)=z_1^n$ and $\tilde t_n^{(1)}(z_1,z_2)=z_2^n$ so that the extremal-like functions in (\ref{vkkt}) are given by $\tilde V_{K_{\epsilon}}^{(0)}(z_1,z_2)= \log|z_1|$ and  $\tilde V_{K_{\epsilon}}^{(1)}(z_1,z_2)= \log|z_2|$. Hence we do have the equality 
$$V_{K_{\epsilon}}(z_1,z_2)=\max[\tilde V_{K_{\epsilon}}^{(0)}(z_1,z_2),\tilde V_{K_{\epsilon}}^{(1)}(z_1,z_2)] \ \hbox{on} \  A_{\epsilon}\setminus K.$$
Moreover, clearly we can still define and compute Robin constants associated to these directions (which we continue to denote as $\lambda_1$ and $\lambda_2$); in this case we have
$\rho_{K_{\epsilon}}(\lambda_1) =\rho_{K_{\epsilon}}(\lambda_2)=0$. 

Note that this function $ \max[ \log^+|z_1|, \log^+|z_2|]$ equals $V_K(z_1,z_2)$ on $A$ where $A:=\{(z_1,z_2)\in \C^2: z_1 z_2 = 0 \}$ and $K:=A\cap B$ is the union of the unit disks in the $z_1$ and $z_2$ planes. Here $A$ is reducible. To get an example where the Robin constants are different, we replace $B$ by a closed bidisk 
$$B_r:=\{(z_1,z_2): |z_1|\leq r_1, \ |z_2|\leq r_2\}$$
where $r_1 \not = r_2$. Now $K^r:=A\cap B_r$ is the union of the disks $\{(z_1,0):|z_1|\leq r_1\}$ and $\{(0,z_2):|z_2|\leq r_2\}$ so that 
$$V_{K^r}(z_1,z_2)=\max[\log^+ |z_1|/r_1, \log^+ |z_2|/r_2], \ (z_1,z_2)\in A$$
and $\rho_{K^r}(\lambda_1) =-\log r_1$ while $\rho_{K^r}(\lambda_2)=-\log r_2$. This result on the directional Robin constants also follows from Proposition 4.7 in \cite{M}.

\end{example}


\begin{thebibliography}{GGK2}

\bibitem{BM} W. Baleikorocau and S. Ma'u, Chebyshev constants, transfinite diameter, and computation on algebraic curves, Comp. Methods Funct. Theory, \textbf{15}, (2015), 291-322.
 
 \bibitem{BT} E. Bedford and B. A. Taylor, Plurisubharmonic functions with logarithmic singularities, Ann. Inst. Fourier, Grenoble, \textbf{38}, (1988), no. 4, 133-171.
 
 \bibitem{B98} T. Bloom, Some applications of the Robin function to multivariable approximation theory, JAT,  \textbf{92},(1998), 1-21.
 
 \bibitem{BL} T. Bloom and N. Levenberg, Distribution of nodes on algebraic curves in $\C^N$, Ann. Inst. Fourier, Grenoble, \textbf{53}, (2003), no. 5, 1365-1385.
 
 \bibitem{BLM} T. Bloom, N. Levenberg and S. Ma'u, Robin functions and extremal function,  Annales Polonici Math., \textbf{80}, (2003), 55-84.
 
 \bibitem{HM} J. Hart and S. Ma'u, Chebyshev and Robin constants on algebraic curves, Annales Polonici Math., \textbf{115}, (2015), no. 2, 101-121.
 
 \bibitem{L} N. Levenberg, Capacities in several complex variables, PhD thesis, U. of Michigan, 1984.
 
 \bibitem{M} S. Ma'u, Chebyshev constants and transfinite diameter on algebraic curves in $\C^2$. Indiana Univ. Math. J., \textbf{60}, (2011), no. 5, 1767-1796.
 
 \bibitem{S} A. Sadullaev, Estimates of polynomials on analytic sets, Izv. Akad. Nauk SSSR Ser. Mat. \textbf{46}, (1982), 524-534. 
 
 \bibitem{Z} V. P. Zaharjuta, Transfinite diameter, Chebyshev constants, and capacity for compacta in $\C^n$, Math. USSR Sbornik, \textbf{25}, (1975), no. 3, 350-364. 
 
 
\end{thebibliography}
\end{document}